\documentclass[12pt,a4paper]{article}

\usepackage{amsmath,amssymb,amsthm}
\usepackage{graphicx}

\numberwithin{equation}{section}
\newtheorem{theorem}{Theorem}[section]
\newtheorem{lemma}[theorem]{Lemma}
\newtheorem{proposition}[theorem]{Proposition}

\theoremstyle{definition}
\newtheorem{definition}[theorem]{Definition}
\theoremstyle{remark}
\newtheorem{remark}[theorem]{Remark}

\newcommand{\C}{\mathbb C}
\newcommand{\N}{\mathbb N}
\newcommand{\D}{\mathcal D}
\newcommand{\B}{\mathcal B}
\newcommand{\Hankel}{\mathcal H}

\newcommand{\dd}{\,\mathrm d}
\newcommand{\Log}{\operatorname{Log}}
\newcommand{\Arg}{\operatorname{Arg}}
\newcommand{\dist}{\operatorname{dist}}
\newcommand{\im}{\operatorname{Im}}
\newcommand{\re}{\operatorname{Re}}

\title{\textbf{Rigorous Asymptotic Analysis of 3-Noncrossing Skeleton Diagrams}}
\author{Yangyang Zhao}
\date{August 2026}

\begin{document}
\maketitle

\begin{abstract}
We give a complete rigorous asymptotic analysis of the generating functions of
3-noncrossing skeleton matchings and canonical 3-noncrossing skeleton diagrams.
Let $F_3$ be the ordinary generating function of 3-noncrossing matchings, and
let $S(y)=\sum_{n\geq 0}S(n)y^n$ be determined by
$S(zF_3(z)^2)=F_3(z)$. The proof is deliberately ordered to avoid circularity.
First, Lagrange inversion, a Stieltjes representation of $F_3$, exact
cut-boundary estimates, and a moving horizontal Hankel contour give
$S(n)\sim 24(\pi A^5)^{-1}\sigma^{-n}n^{-5}$ independently of any
$\Delta$-analyticity of $S$. This estimate supplies boundary regularity of $S$
and $S'$. We then prove a global biholomorphic inversion theorem, continuation
across every nonprincipal point of the convergence circle, and a logarithmically
perturbed sectorial inverse theorem. A complete disk-chain and monodromy
argument yields a single-valued continuation to a standard $\Delta$-domain. At
the principal singularity,
$S(y)=Q_4(u)-(\pi A^5)^{-1}u^4\log u+O(u^5(1+|\log u|))$, where
$u=1-y/\sigma$. Finally, the canonical composition
$S_3^{[4]}(z)=(1-z)(S(\vartheta(z))-1-\vartheta(z))$ is shown to be
$\Delta$-analytic at its unique dominant singularity
$\eta=0.49340718057613087519\ldots$, and
$[z^n]S_3^{[4]}(z)\sim 7892.16205625817\ldots n^{-5}\eta^{-n}$. All analytic-continuation and inversion arguments are carried out on
explicitly specified single-valued domains.
\end{abstract}

\tableofcontents

\section{Introduction and combinatorial background}

A diagram on $[N]=\{1,\ldots,N\}$ is represented by vertices on a horizontal
line and arcs $(i,j)$ in the upper half-plane.  Each vertex is incident with at most one arc; a matching has no isolated vertices.
It is $k$-noncrossing when it
contains no $k$ mutually crossing arcs
\[
  i_1<i_2<\cdots<i_k<j_1<j_2<\cdots<j_k.
\]
For the standard crossing and nesting terminology for matchings and
partitions, see \cite{ChenEtAl}.

A \emph{stack} is a maximal set of parallel consecutive arcs
$(i+s,j-s)$, $0\le s<\ell$.  The \emph{core} is obtained by collapsing
each stack to one arc.  A \emph{stem} is a maximal chain of stacks that
becomes one stack after isolated vertices are deleted.

The \emph{dependency graph} of a diagram has one vertex for every arc and joins
two vertices exactly when the corresponding arcs cross.  A skeleton diagram is
a diagram whose core has no noncrossing arc and whose dependency graph is
connected.  Equivalently, after isolated vertices and parallel arcs in stacks
are removed, every remaining arc is crossed and all arcs belong to a single
crossing component.  A $V_k$-shape is a $k$-noncrossing matching whose stacks
all have length one; the shape of a skeleton is obtained by replacing every
stem by one arc and then deleting all isolated vertices.  A diagram is
irreducible if no vertex separates its arcs into a left and a right nonempty
part.

These objects occur as the outer crossing framework in pseudoknot RNA
structures.  In the folding algorithm of Huang, Peng, and Reidys
\cite{HuangPengReidys}, the leaves of the skeleton tree are skeleton
structures, so their enumeration controls an exponential part of the
algorithmic complexity.  Sections 4.1--4.3 of Zhao's dissertation
\cite[Chapter~4, Sections~4.1--4.3, pp.~81--90]{ZhaoThesis} established the
relevant combinatorial decompositions.
If
$F_3(z)$ counts 3-noncrossing matchings by number of arcs, $\operatorname{Irr}(z)$
counts irreducible 3-noncrossing matchings, and $S(n)$ counts 3-noncrossing
skeleton matchings with $n$ arcs (with the harmless conventions
$S(0)=S(1)=1$), then
\[
 1+\sum_{n\geq1}S(n)F_3(z)^{2n-1}z^n=\operatorname{Irr}(z),
 \qquad
 \operatorname{Irr}(z)=2-\frac1{F_3(z)}.
\]
Consequently,
\begin{equation}\label{eq:implicit-intro}
  S\bigl(zF_3(z)^2\bigr)=F_3(z),
  \qquad S(y)=\sum_{n\geq0}S(n)y^n.
\end{equation}

For canonical 3-noncrossing skeleton diagrams with minimum stack length three
and minimum arc length four, we count nonempty skeletons whose first and last vertices are paired; external sequences of isolated vertices are excluded.
The two-step inflation of a skeleton shape gives
\begin{equation}\label{eq:canonical-intro}
 S_3^{[4]}(z)=(1-z)G(\vartheta(z)),
 \qquad G(y)=S(y)-1-y,
\end{equation}
where
\begin{equation}\label{eq:theta-def-intro}
 w_0(z)=\frac{z^4}{1-z^2+z^6},
 \qquad
 \vartheta(z)=\left(\frac{z\sqrt{w_0(z)}}{1-z}\right)^2
 =\frac{z^6}{(1-z)^2(1-z^2+z^6)}.
\end{equation}

Write $\mathcal{IS}(n)$ for the set of skeleton shapes with $n$ arcs and
$I_{\rm sh}(x)=\sum_{n\ge2}|\mathcal{IS}(n)|x^n$.
Inflating each shape arc into a stem and inserting isolated vertices in the
remaining interior gaps yields

\[
 \sum_{n\ge2}\sum_{\gamma\in\mathcal{IS}(n)}S_\gamma(z)
 =(1-z)\,I_{\rm sh}\!\left(
   \frac{z^6}{1-2z+2z^3-z^4-2z^7+z^8}\right),
\]
and the relation
$I_{\rm sh}(x)=\sum_{n\geq2}S(n)(x/(1+x))^n$ reduces this expression to
\eqref{eq:canonical-intro}--\eqref{eq:theta-def-intro}.

The analytic strategy transfers the singular behaviour of the D-finite matching
series through an implicit functional equation and then through a nontrivial
canonical composition.  The proof treats explicitly the global choice of the
inverse branch, exclusion of critical boundary preimages, single-valued gluing
on a genuine $\Delta$-domain, and the contribution made by the logarithmic term
of the inverse map itself.  It is organized in the logically noncircular order
\begin{align*}
 \text{independent coefficient asymptotics}
 &\Longrightarrow \text{boundary regularity}\\
 &\Longrightarrow \Delta\text{-continuation}\\
 &\Longrightarrow \text{singular expansion and transfer}.
\end{align*}

\section{Notation and principal results}

Write $F=F_3$.  The explicit hypergeometric expression is
\begin{equation}\label{eq:F-def}
 F(z)=\frac1{4z^2}\left(
  1+6z-{}_2F_1\!\left(-\frac12,-\frac12;1;16z\right)
  -2z\,{}_2F_1\!\left(-\frac12,\frac32;3;16z\right)
 \right),
\end{equation}
with removable value $F(0)=1$.  The matching enumeration of \cite[Section~5]{ChenEtAl} gives the Catalan determinant \eqref{eq:catalan-det}; Appendix~\ref{app:catalan} verifies its equality with the coefficients of \eqref{eq:F-def}.
Its Taylor expansion begins
\[
 F(z)=1+z+3z^2+14z^3+84z^4+\cdots.
\]
Set
\begin{equation}\label{eq:constants}
 \rho:=\frac1{16},\qquad F_0:=F(\rho),\qquad
 \sigma:=\rho F_0^2,\qquad
 A:=1+\frac{2\rho F'(\rho)}{F_0}.
\end{equation}
Direct evaluation gives
\begin{equation}\label{eq:constants-exact}
 F_0=88-\frac{4096}{15\pi}=\frac{8(165\pi-512)}{15\pi},
 \quad
 A=\frac{1280-405\pi}{165\pi-512}>0,
 \quad
 \sigma=\frac{4(165\pi-512)^2}{225\pi^2}.
\end{equation}
The positivity asserted here is exact, rather than numerical.  Indeed, the
classical rational bounds
\[
 \frac{333}{106}<\pi<\frac{22}{7}
\]
give
\[
 165\pi-512>\frac{673}{106}>0,
 \qquad
 1280-405\pi>\frac{50}{7}>0.
\]
Numerically,
\begin{align*}
 A&=1.203085107635387\ldots,\\
 \sigma&=0.0729243577550381\ldots,\\
 \sigma^{-1}&=13.71283931439099\ldots.
\end{align*}
All decimal values displayed in this article are rounded from directed interval
computations.  One reproducible certification starts with Machin's identity
\[
 \pi=16\arctan\frac15-4\arctan\frac1{239}
\]
and bounds each arctangent by consecutive partial sums of its alternating
series.  Rational interval arithmetic in \eqref{eq:constants-exact}, followed
by interval Newton steps for the equation defining $\eta$, certifies every
printed digit; none of the inequalities below relies only on a floating-point
comparison.
Define
\begin{equation}\label{eq:Phi-D}
 \Phi(z):=zF(z)^2,\qquad
 \D:=\C\setminus[\rho,\infty),\qquad
 \B_\sigma:=\{y\in\C:|y|<\sigma\}.
\end{equation}

The main conclusions are the following.

\begin{theorem}[Skeleton-matching coefficients]\label{thm:S-coeff-main}
As $n\to\infty$,
\begin{equation}\label{eq:S-coeff-main}
 S(n)\sim \frac{24}{\pi A^5}\,\sigma^{-n}n^{-5}
 =\frac{24(165\pi-512)^5}{\pi(1280-405\pi)^5}\,
   \sigma^{-n}n^{-5}.
\end{equation}
The proof in Section~\ref{sec:independent-coeff} is independent of any
$\Delta$-analyticity of $S$.
\end{theorem}

\begin{theorem}[$\Delta$-continuation and singular expansion]\label{thm:S-delta-main}
There exist $\epsilon>0$ and $0<\phi<\pi/2$ such that $S$ has a unique
single-valued holomorphic continuation to
\[
 \Delta(\sigma,\epsilon,\phi)
 =\{y:|y|<\sigma+\epsilon,\ y\ne\sigma,\
       |\Arg(y-\sigma)|>\phi\}.
\]
Moreover, for a polynomial $Q_4$ of degree at most four, uniformly in every
fixed smaller $\Delta$-sector,
\begin{align}\label{eq:S-sing-main}
 S(y)={}&Q_4\!\left(1-\frac y\sigma\right)
 -\frac1{\pi A^5}\left(1-\frac y\sigma\right)^4
       \Log\!\left(1-\frac y\sigma\right)\notag\\
 &+O\!\left(\left(1-\frac y\sigma\right)^5
 \left[1+\left|\Log\!\left(1-\frac y\sigma\right)\right|\right]\right).
\end{align}
The point $\sigma$ is the unique singularity on the convergence circle.
\end{theorem}

\begin{theorem}[Canonical skeleton diagrams]\label{thm:canonical-main}
Let $\eta$ be the unique solution in $(0,1/2)$ of $\vartheta(\eta)=\sigma$.  Then
\[
 \eta=0.49340718057613087519\ldots
\]
is the unique dominant singularity of $S_3^{[4]}$, this generating function is
$\Delta$-analytic at $\eta$, and
\begin{equation}\label{eq:canonical-coeff-main}
 [z^n]S_3^{[4]}(z)\sim C'n^{-5}\eta^{-n},
 \qquad C'=7892.16205625817\ldots.
\end{equation}
An exact expression for $C'$ is given in
Theorem~\ref{thm:canonical-coeff}.
\end{theorem}

\section{Independent coefficient asymptotics for $S$}\label{sec:independent-coeff}

\subsection{Lagrange--B\"urmann inversion}

Let $X(y)$ be the formal compositional inverse of $\Phi$ at the origin.  Since
$\Phi(0)=0$ and $\Phi'(0)=F(0)^2=1$, this inverse exists uniquely, and
\[
 X(y)=yF(X(y))^{-2},\qquad S(y)=F(X(y)).
\]
The Lagrange--B\"urmann formula \cite{Gessel} states that if
$X=y\varphi(X)$, then, for
$n\geq1$,
\[
 [y^n]H(X(y))=\frac1n[z^{n-1}]H'(z)\varphi(z)^n.
\]
Taking $H=F$ and $\varphi=F^{-2}$ gives
\begin{align}
 S(n)&=\frac1n[z^{n-1}]F'(z)F(z)^{-2n}\notag\\
 &=\frac1{n(1-2n)}[z^{n-1}]\frac{\dd}{\dd z}F(z)^{1-2n}
 =-\frac1{2n-1}[z^n]F(z)^{-(2n-1)}.\label{eq:LB-exact}
\end{align}
Put $m:=2n-1$.  If $\gamma$ is a small positively oriented circle about the
origin, Cauchy's formula yields
\begin{equation}\label{eq:Cauchy-start}
 S(n)=-\frac1m\frac1{2\pi i}\int_\gamma F(z)^{-m}z^{-n-1}\dd z.
\end{equation}
This step uses only formal inversion and local holomorphy at the origin.

\subsection{A Stieltjes continuation and zero-freeness of $F$}

Let $C_n=(n+1)^{-1}\binom{2n}{n}$ be the $n$th Catalan number.  Direct
coefficient extraction from \eqref{eq:F-def}, detailed in
Appendix~\ref{app:catalan}, gives
\begin{equation}\label{eq:catalan-det}
 f_n:=[z^n]F(z)=C_nC_{n+2}-C_{n+1}^2.
\end{equation}
On $[0,4]$ introduce the positive probability measure
\begin{equation}\label{eq:catalan-measure}
 \dd\mu(u)=\frac1{2\pi}\sqrt{\frac{4-u}{u}}\,
             \mathbf1_{(0,4)}(u)\dd u.
\end{equation}
The substitution $u=4s$ and the beta integral give
\[
 \int_0^4u^n\dd\mu(u)
 =\frac{2\cdot4^n}{\pi}B\!\left(n+\frac12,\frac32\right)
 =\frac1{n+1}\binom{2n}{n}=C_n.
\]
Substituting this moment representation into \eqref{eq:catalan-det} yields
\[
 f_n=\frac12\int_0^4\!\int_0^4(uv)^n(u-v)^2\dd\mu(u)\dd\mu(v).
\]
Let $\nu$ be the pushforward under $(u,v)\mapsto uv$ of the positive finite
measure $\tfrac12(u-v)^2\dd\mu(u)\dd\mu(v)$.  Then
\begin{equation}\label{eq:moment-nu}
 \operatorname{supp}\nu\subset[0,16],
 \qquad f_n=\int_0^{16}t^n\dd\nu(t).
\end{equation}

\begin{lemma}[Stieltjes continuation]\label{lem:stieltjes}
The function
\begin{equation}\label{eq:stieltjes}
 \widetilde F(z):=\int_0^{16}\frac{\dd\nu(t)}{1-zt}
\end{equation}
is holomorphic on $\D=\C\setminus[\rho,\infty)$, and
\[
 \widetilde F^{(k)}(z)=k!\int_0^{16}\frac{t^k}{(1-zt)^{k+1}}\dd\nu(t),
 \qquad k\geq0.
\]
It agrees there with the principal hypergeometric continuation of $F$.
\end{lemma}

\begin{proof}
Fix $z_0\in\D$ and choose a closed disk $K\Subset\D$ about $z_0$.  The
continuous function $(z,t)\mapsto|1-zt|$ has a strictly positive minimum on
$K\times[0,16]$: equality to zero would give $z=1/t\in[\rho,\infty)$.
Thus the kernel and all of its $z$-derivatives are uniformly bounded on that
compact product.  Since $\nu$ is finite, dominated convergence justifies
differentiation under the integral sign to every order.  Hence
$\widetilde F\in\mathcal O(\D)$.

For $|z|<\rho$, the geometric series is uniform in $t\in[0,16]$, and
\[
 \widetilde F(z)=\sum_{n\geq0}z^n\int_0^{16}t^n\dd\nu(t)
 =\sum_{n\geq0}f_nz^n=F(z).
\]
The principal ${}_2F_1$ branch is holomorphic on
$\C\setminus[1,\infty)$; therefore \eqref{eq:F-def} is holomorphic on
$\D$.  The identity theorem on the connected domain $\D$ identifies the two
continuations.
\end{proof}

\begin{lemma}[Zero-freeness]\label{lem:zero-free}
$F(z)\ne0$ for every $z\in\D$.
\end{lemma}

\begin{proof}
If $\im z>0$, then for $t>0$,
\[
 \im\frac1{1-zt}=\frac{t\im z}{|1-zt|^2}>0.
\]
Since $1=f_1=\int_{(0,16]}t\dd\nu(t)\le16\nu((0,16])$, the measure $\nu$ has positive mass in $(0,16]$.  Thus $\im F(z)>0$.  The lower half-plane is analogous.  If $z<\rho$
is real, then $1-zt>0$ for $0\leq t\leq16$, and the integral is strictly
positive.  These three cases exhaust $\D$.
\end{proof}

In particular, $F^{-m}$ is holomorphic throughout $\D$ for every positive
integer $m$.

\subsection{The endpoint expansion and exact constants}

Put
\begin{equation}\label{eq:u-coordinate}
 u:=1-\frac z\rho,\qquad z=\rho(1-u),
\end{equation}
and use the principal value of $\Log u$ on
$\C\setminus(-\infty,0]$.  The integer-difference connection formula
\cite[\S15.8, Eq.~(15.8.10)]{OldeDaalhuis}, applied to the two
hypergeometric functions in
\eqref{eq:F-def}, gives uniformly in every fixed sector
$0<|u|<\delta$, $|\Arg u|\leq\pi-\epsilon$,
\begin{align}
 {}_2F_1\!\left(-\frac12,-\frac12;1;1-u\right)
 &=P_1(u)+\left(-\frac{u^2}{8\pi}-\frac{3u^3}{32\pi}
   -\frac{75u^4}{1024\pi}+O(u^5)\right)\Log u,\label{eq:H1-local}\\
 {}_2F_1\!\left(-\frac12,\frac32;3;1-u\right)
 &=P_2(u)+\left(\frac{u^2}{\pi}+\frac{7u^3}{4\pi}
   +\frac{315u^4}{128\pi}+O(u^5)\right)\Log u,
   \label{eq:H2-local}
\end{align}
where $P_1,P_2$ are holomorphic at zero.  After multiplication by the rational
prefactors in \eqref{eq:F-def}, the $u^2\Log u$ and $u^3\Log u$ terms cancel.
Thus there is a function $P$ holomorphic at zero such that
\begin{equation}\label{eq:F-endpoint}
 F(\rho(1-u))=P(u)+Bu^4\Log u
   +O\bigl(u^5(1+|\Log u|)\bigr),
 \qquad B=-\frac1\pi,
\end{equation}
uniformly in such sectors, with
\begin{equation}\label{eq:P-data}
 P(0)=F_0,\qquad P'(0)=-\rho F'(\rho).
\end{equation}
The coefficient calculation is recorded in Appendix~\ref{app:logcoeff}.

The convergent expansions used above also give the uniformity needed when a
point approaches the endpoint arbitrarily close to the cut.

\begin{lemma}[Uniform bounds on a slit disk]\label{lem:slit-uniform}
The functions $P$ in \eqref{eq:F-endpoint} and $q_0$ can be chosen
holomorphic on $|u|<1$ so that
\begin{equation}\label{eq:exact-Pq}
 F(\rho(1-u))=P(u)+q_0(u)\Log u,
 \qquad q_0(u)=Bu^4+O(u^5),\quad B=-\frac1\pi.
\end{equation}
For every fixed $0<r<1$ there is $C_r<\infty$ such that
\[
 \left|F(\rho(1-u))-P(u)-Bu^4\Log u\right|
 \le C_r|u|^5(1+|\log|u||)
\]
whenever $0<|u|\le r$ and $|\Arg u|<\pi$.  In particular,
$F(z)\to F_0$ and $\Phi(z)\to\sigma$ whenever $z\to\rho$ within $\D$.
\end{lemma}

\begin{proof}
Substituting the convergent expansions of Appendix~\ref{app:logcoeff} into
\eqref{eq:F-def} gives \eqref{eq:exact-Pq} on the slit disk.  The only
additional rational denominator is $(1-u)^2$, which is nonzero for
$|u|<1$.  The cancellation computed in that appendix gives
$q_0(u)=Bu^4+u^5b(u)$ with $b$ holomorphic on the disk.  Boundedness of $b$
on $|u|\le r$ and $|\Log u|\le|\log|u||+\pi$ prove the estimate.  The
limits follow from $P(0)=F_0$ and $x^4(1+|\log x|)\to0$.
\end{proof}

Gauss's evaluation at one gives
\[
 {}_2F_1\!\left(-\frac12,-\frac12;1;1\right)=\frac4\pi,
 \qquad
 {}_2F_1\!\left(-\frac12,\frac32;3;1\right)=\frac{32}{15\pi}.
\]
Differentiating with respect to the hypergeometric argument
and applying Gauss's formula gives $H_1'(1)=1/\pi$ and
$H_2'(1)=-8/(5\pi)$, where $H_1,H_2$ denote the two functions just displayed
\cite[Eqs.~(15.5.1), (15.4.20)]{DLMF}.  Substitution
into \eqref{eq:F-def} yields
\[
 F_0=88-\frac{4096}{15\pi},
 \qquad F'(\rho)=-2432+\frac{114688}{15\pi},
\]
and hence \eqref{eq:constants-exact}.

\subsection{Strict separation on the two sides of the cut}

We next prove the global estimate needed to move the Cauchy contour:
\begin{equation}\label{eq:cut-separation-goal}
 r|F_\pm(r)|^2>\sigma\qquad(r>\rho),
\end{equation}
where $F_\pm(r)$ denotes the boundary value of $F$ from the upper or lower
side of the cut:
\begin{equation}\label{eq:side-boundary-values}
 F_\pm(r):=
 \lim_{\substack{z\to r\\ \pm\im z>0}}F(z)
 =\lim_{\epsilon\downarrow0}F(r\pm i\epsilon),
 \qquad r>\rho.
\end{equation}
The existence of these one-sided limits follows from the separate local
holomorphic continuations of $F$ from the upper and lower half-planes furnished
by the connection formulae below (see Lemma~\ref{lem:cut-neighborhood} for the open-neighborhood construction);
the two local branches are not asserted to coincide.

Use parameter notation for the complete elliptic integrals
\[
 K(q)=\int_0^{\pi/2}\frac{\dd\theta}{\sqrt{1-q\sin^2\theta}},
 \qquad
 E(q)=\int_0^{\pi/2}\sqrt{1-q\sin^2\theta}\,\dd\theta.
\]
Thus, if $\mathsf K(k),\mathsf E(k)$ denote the modulus notation used in
DLMF, then $K(q)=\mathsf K(\sqrt q)$ and $E(q)=\mathsf E(\sqrt q)$; the
integral definitions are the complete cases of DLMF
\cite[\S19.2(ii), Eqs.~(19.2.4), (19.2.5), and (19.2.8)]{DLMF}.
The corresponding printed Handbook treatment is
\cite[pp.~485--522]{Carlson}.
The hypergeometric representations in DLMF
\cite[\S19.5, Eqs.~(19.5.1)--(19.5.2)]{DLMF}, followed respectively by a
coefficient comparison and by Euler's integral representation
\cite[\S15.6, Eq.~(15.6.1)]{DLMF}, give
\begin{align}
 {}_2F_1\!\left(-\frac12,-\frac12;1;z\right)
 &=\frac2\pi\bigl(2E(z)-(1-z)K(z)\bigr),\label{eq:ell-H1}\\
 {}_2F_1\!\left(-\frac12,\frac32;3;z\right)
 &=\frac{16}{15\pi z^2}
 \bigl(2(z^2-z+1)E(z)-(z^2-3z+2)K(z)\bigr).
 \label{eq:ell-H2}
\end{align}
The coefficient and integral reductions, including removal of the apparent
singularity at $z=0$ in \eqref{eq:ell-H2}, are given in
Appendix~\ref{app:elliptic-algebra}.  Substitution into \eqref{eq:F-def} and
collection of the $E$- and $K$-terms, also recorded in that appendix, gives
\begin{equation}\label{eq:F-elliptic}
 F\!\left(\frac z{16}\right)=\frac8{15\pi z^3}
 \left[\pi(45z^2+120z)-32(z^2+14z+1)E(z)
 -32(z-1)(7z+1)K(z)\right].
\end{equation}

For $r>\rho$, put $q=(16r)^{-1}\in(0,1)$.  The reciprocal-parameter connection
formulae are the parameter-notation form of DLMF
\cite[\S19.7(i), Eq.~(19.7.3) and the sign convention following it]{DLMF}:
\begin{align}
 K_\pm(1/q)&=\sqrt q\,[K(q)\pm iK(1-q)],\label{eq:K-boundary}\\
 E_\pm(1/q)&=q^{-1/2}\bigl[E(q)-(1-q)K(q)
 \mp i(E(1-q)-qK(1-q))\bigr].\label{eq:E-boundary}
\end{align}
Here the reversal between the DLMF sign attached to the reciprocal modulus and
the sign of the boundary point $1/q\pm i0$ is explained explicitly in
Appendix~\ref{app:elliptic-algebra}.
Inserting these into \eqref{eq:F-elliptic} yields
\begin{equation}\label{eq:F-boundary-UV}
 F_\pm(r)=\frac{8\sqrt q}{15\pi}\,[U(q)\pm32iV(q)],
\end{equation}
where
\begin{align}
 U(q)={}&\pi\sqrt q(45+120q)
 +32(1-q)(1+7q)K(q)-32(1+14q+q^2)E(q),\label{eq:U-def}\\
 V(q)={}&(1+14q+q^2)E(1-q)-8q(1+q)K(1-q).\label{eq:V-def}
\end{align}
Therefore
\begin{equation}\label{eq:modulus-UV}
 r|F_\pm(r)|^2=\frac4{225\pi^2}\bigl(U(q)^2+1024V(q)^2\bigr).
\end{equation}
As $q\uparrow1$,
\[
 U(q)\longrightarrow U_*:=165\pi-512>0,\qquad V(q)\longrightarrow0,
 \qquad \sigma=\frac{4U_*^2}{225\pi^2}.
\]
The algebra leading to \eqref{eq:F-boundary-UV}--\eqref{eq:modulus-UV} and
these endpoint limits is included in Appendix~\ref{app:elliptic-algebra}.

For later contour estimates, the reciprocal formulae are needed on open
one-sided neighborhoods, not only as boundary values.

\begin{lemma}[Holomorphic side branches on compact cut intervals]
\label{lem:cut-neighborhood}
For every compact interval $I\subset(\rho,\infty)$ there is an open
neighborhood $W$ of $I$ and two holomorphic functions
$\widetilde F_\pm$ on $W$ such that
\[
 \widetilde F_+(z)=F(z)\quad(z\in W,\ \im z>0),\qquad
 \widetilde F_-(z)=F(z)\quad(z\in W,\ \im z<0).
\]
Their restrictions to $I$ are the side values in
\eqref{eq:F-boundary-UV}.  Moreover, for some $C_I$ and all sufficiently
large $n$,
\[
 \sup_{r\in I}|F(r\pm i/n)-F_\pm(r)|\le C_I/n.
\]
\end{lemma}

\begin{proof}
Put $Q(z)=(16z)^{-1}$.  Since $Q(I)\Subset(0,1)$, choose $W$ so small
that both $Q(W)$ and $1-Q(W)$ remain in the right half-plane and away from
the elliptic cut $[1,\infty)$.  Take all square roots below to be positive
on $I$ and define on $W$
\begin{align*}
 \mathcal K_\pm(z)
 &=\sqrt{Q(z)}\,[K(Q(z))\pm iK(1-Q(z))],\\
 \mathcal E_\pm(z)
 &=Q(z)^{-1/2}\{E(Q(z))-(1-Q(z))K(Q(z))\\
 &\hspace{35mm}\mp i[E(1-Q(z))-Q(z)K(1-Q(z))]\}.
\end{align*}
These functions are holomorphic on $W$.  Apply DLMF
\cite[Eq.~(19.7.3)]{DLMF} with $k=\sqrt{Q(z)}$, for which $\re k>0$.
Because the sign of $\im Q(z)$ is opposite to that of $\im z$, the formula
identifies $\mathcal K_+,\mathcal E_+$ with $K(16z),E(16z)$ on the upper
half of $W$, and the minus versions on the lower half.  Substitution into
\eqref{eq:F-elliptic}, with its variable replaced by $16z$, defines
$\widetilde F_\pm$.  Thus the identities hold on open half-neighborhoods
and genuinely continue the principal function $F$.

On a smaller closed strip about $I$ contained in $W$, the derivatives
$\widetilde F_\pm'$ are bounded.  Integrating along the vertical segment
from $r$ to $r\pm i/n$ proves the asserted uniform bound.
\end{proof}

It remains to prove
\begin{equation}\label{eq:UV-target}
 U(q)^2+1024V(q)^2>U_*^2\qquad(0<q<1).
\end{equation}

\paragraph{The range $1/8\le q<1$.}
After conversion from the DLMF modulus $k$ to the present parameter $q=k^2$,
the derivative formulae
\[
 K'(q)=\frac{E(q)}{2q(1-q)}-\frac{K(q)}{2q},
 \qquad E'(q)=\frac{E(q)-K(q)}{2q}
\]
follow from DLMF \cite[\S19.4(i), Eqs.~(19.4.1)--(19.4.2)]{DLMF}.  Direct
differentiation and collection of the $K$- and $E$-terms, detailed in
Appendix~\ref{app:elliptic-algebra},
give
\begin{equation}\label{eq:Uprime}
 U'(q)=-\frac5{2\sqrt q}H(q),
\end{equation}
where
\begin{equation}\label{eq:H-def}
 H(q)=32\sqrt q\bigl((q+7)E(q)-4(1-q)K(q)\bigr)-9\pi(8q+1).
\end{equation}
Set $q=t^2$ and
\[
 a_n:=\frac{\binom{2n}{n}^2}{16^n}>0.
\]
Using the parameter versions of DLMF
\cite[\S19.5, Eqs.~(19.5.1)--(19.5.2)]{DLMF}, namely
\[
 K(t^2)=\frac\pi2\sum_{n\ge0}a_nt^{2n},
 \qquad
 E(t^2)=\frac\pi2\left(1-\sum_{n\ge1}\frac{a_n}{2n-1}t^{2n}\right),
\]
and collecting powers as in Appendix~\ref{app:elliptic-algebra}, one obtains
\begin{equation}\label{eq:H-PT}
 \frac{H(t^2)}\pi=P_H(t)-T(t),
 \quad P_H(t)=-9+48t-72t^2+36t^3,
 \quad T(t)=36\sum_{n\ge2}\frac{a_{n-1}}{n^2(2n-3)}t^{2n+1}.
\end{equation}
All coefficients of $T$ are positive.  Since $0<a_{n-1}\le1$, the series
defining $T(1)$ is dominated by a constant multiple of $\sum_{n\ge2}n^{-3}$
and therefore converges.  Hence, since $2n+1\ge5$,
$T(t)\le t^5T(1)$ for $0<t\le1$.  In the following display $H(1)$ denotes
the limit as $q\uparrow1$.  The defining integrals give $E(q)\to1$, while
\[
 0\le(1-q)K(q)\le \frac\pi2\sqrt{1-q}\longrightarrow0,
\]
because $1-q\sin^2\theta\ge1-q$.  Thus one may pass to the endpoint in
\eqref{eq:H-def} and, by absolute convergence, in \eqref{eq:H-PT}, obtaining
\[
 H(1)=256-81\pi,\qquad P_H(1)=3,\qquad
 T(1)=84-\frac{256}{\pi}<3,
\]
where the last inequality follows from $\pi<22/7<256/81$.  Consequently
\begin{equation}\label{eq:H-lower}
 \frac{H(t^2)}\pi>P_H(t)-3t^5=3(1-t)Q(t),
 \qquad Q(t)=t^4+t^3-11t^2+13t-3.
\end{equation}
For $1/(2\sqrt2)\le t\le1$, $Q(t)>0$.  Indeed,
$Q''(t)=12t^2+6t-22<0$ on $(0,1]$.  On
$[1/(2\sqrt2),1/2]$, $Q'$ is decreasing and
$Q'(t)\ge Q'(1/2)=13/4$, while
\[
 Q\!\left(\frac1{2\sqrt2}\right)=\frac{210\sqrt2-279}{64}>0.
\]
On $[1/2,1]$, concavity puts the minimum at an endpoint, and
$Q(1/2)=15/16$, $Q(1)=1$.  Hence $H(q)>0$ and $U'(q)<0$ for
$1/8\le q<1$.  Therefore
\begin{equation}\label{eq:U-lower}
 U(q)>U_*\qquad(1/8\le q<1),
\end{equation}
which proves \eqref{eq:UV-target} in this range.

\paragraph{The range $0<q\le1/8$.}
The same derivative formulae and the chain rule applied at $1-q$ give
\begin{equation}\label{eq:Vprime}
 V'(q)=\frac52\bigl((q+7)E(1-q)-(5q+3)K(1-q)\bigr).
\end{equation}
The intermediate coefficient calculation is given in
Appendix~\ref{app:elliptic-algebra}.
Define
\[
 D_0(s):=(8-5s)K(s)-(8-s)E(s).
\]
The Maclaurin series and
$a_n/a_{n-1}=((2n-1)/(2n))^2$ yield
\begin{equation}\label{eq:D0-positive}
 D_0(s)=\frac\pi2\sum_{n\ge3}
 \frac{6(n-1)(n-2)}{n(2n-3)}a_{n-1}s^n>0
 \qquad(0<s<1).
\end{equation}
Thus $V'(q)=-(5/2)D_0(1-q)<0$, so $V(q)\ge V(1/8)$ on the present range.

To obtain an explicit lower bound, write
\begin{equation}\label{eq:V-one-eighth}
 V(1/8)=\frac{177}{64}E(7/8)-\frac98K(7/8).
\end{equation}
The convergent complementary-modulus expansions in DLMF
\cite[\S19.12, Eqs.~(19.12.1)--(19.12.3)]{DLMF}, rewritten in the present
parameter $q=(k')^2$, are
\begin{align}
 K(1-q)&=\sum_{n\ge0}a_nq^n(L_q-d_n),\label{eq:K-complementary}\\
 E(1-q)&=1+\sum_{n\ge1}\frac{2n}{2n-1}a_nq^n
 \left(L_q-d_n+\frac1{2n(2n-1)}\right),\label{eq:E-complementary}
\end{align}
where $L_q=\log(4/\sqrt q)$ and $d_n=2(H_{2n}-H_n)$.  At $q=1/8$,
$L_q=(7/2)\log2=:L$.  Since $d_n<2\log2<L$, all remaining terms in
\eqref{eq:E-complementary} are positive.  For completeness, the required
conversion from the DLMF coefficients to $a_n,d_n$, including the index shift and
the sign in \eqref{eq:E-complementary}, is proved in
Appendix~\ref{app:elliptic-algebra}.

We also verify that the $L_q$ coefficients and the nonlogarithmic parts in
these two formulae converge separately.  Put
$b_n=2na_n/(2n-1)$ for $n\ge1$.  The four series are
\[
 \sum_{n\ge0}a_nq^n,\qquad
 -\sum_{n\ge0}a_nd_nq^n,\qquad
 \sum_{n\ge1}b_nq^n,\qquad
 1+\sum_{n\ge1}b_n\left(-d_n+\frac1{2n(2n-1)}\right)q^n.
\]
Here $a_n\le1$, $b_n\le2$, and
$0\le d_n=2\sum_{j=n+1}^{2n}j^{-1}<2\log2$ for $n\ge1$, while $d_0=0$.
Thus the Weierstrass test gives absolute and locally uniform convergence for
$|q|<1$.  Furthermore, Stirling's formula gives
$a_n\sim b_n\sim1/(\pi n)$ and $d_n\to2\log2$.  The four resulting
coefficient asymptotics have nonzero leading constants, so
Cauchy--Hadamard gives radius of convergence exactly one in each case.

The required rational bounds follow without numerical approximation from
\[
 L=\frac72\log2
   =7\sum_{k\ge0}\frac{1}{(2k+1)3^{2k+1}}.
\]
This is the Taylor series
$2\operatorname{arctanh}(1/3)
=\log((1+1/3)/(1-1/3))=\log2$.
Keeping the terms $k=0,1,2,3$ gives
\[
 S_4:=\frac{26528}{10935}<L,
 \qquad
 0<L-S_4
 <\frac{7}{8\cdot3^9}.
\]
Here the upper bound follows by replacing every $2k+1$, $k\ge4$, by $9$
and summing the resulting geometric series of ratio $1/9$.  Direct reduction
of fractions gives
\[
 S_4-\frac{48517}{20000}=\frac{5321}{43740000}>0,
 \qquad
 \frac{12131}{5000}-S_4-\frac{7}{8\cdot3^9}
 =\frac{9049}{49207500}>0.
\]
Consequently
\begin{equation}\label{eq:L-bounds}
 L_-:=\frac{48517}{20000}<L<\frac{12131}{5000}=:L_+.
\end{equation}
Keeping the $n=1$ term in $E$ and bounding the tail of $K$ by $a_n\le1$ gives
\begin{align}
 E(7/8)&>1+\frac1{16}\left(L_--\frac12\right),\label{eq:E78-lower}\\
 K(7/8)&<L_++\frac{L_+-1}{32}+\frac{L_+}{56}.
 \label{eq:K78-upper}
\end{align}
The positivity and tail estimates behind these two inequalities are written
out in Appendix~\ref{app:elliptic-algebra}.
Substitution into \eqref{eq:V-one-eighth} yields
\[
 V(1/8)>\frac{5532789}{20480000}.
\]
On the other hand, using $\pi<22/7$,
\[
 \frac{U_*}{32}=\frac{165\pi-512}{32}<\frac{23}{112},
 \qquad
 \frac{5532789}{20480000}-\frac{23}{112}
 =\frac{9289523}{143360000}>0.
\]
Hence
\begin{equation}\label{eq:V-lower}
 32V(q)\ge32V(1/8)>U_*\qquad(0<q\le1/8).
\end{equation}
Together, \eqref{eq:U-lower} and \eqref{eq:V-lower} prove the following.

\begin{theorem}[Cut-boundary modulus separation]\label{thm:cut-separation}
For every $r>\rho=1/16$,
\[
 r|F_\pm(r)|^2>\rho F(\rho)^2=\sigma.
\]
\end{theorem}

\subsection{Infinity estimates and the moving horizontal Hankel contour}

We use a moving Hankel deformation adapted to the present cut-plane problem;
for the classical Hankel-contour method in asymptotic analysis, see
\cite{Olver}.

The limiting connection formula at infinity for an integral parameter
difference is DLMF
\cite[\S15.8(ii), Eq.~(15.8.8)]{DLMF}.  Applying it to both
hypergeometric functions in \eqref{eq:F-def} gives, uniformly as
$|z|\to\infty$ with $z\in\D$,
\begin{equation}\label{eq:F-infinity}
 F(z)=\frac{64}{15\pi}(-z)^{-1/2}+\frac3{2z}
 +O\bigl(|z|^{-3/2}(1+\log|z|)\bigr),
\end{equation}
where the square-root branch is compatible with $\D$.  Consequently,
\begin{equation}\label{eq:Phi-infinity}
 |z||F(z)|^2=\frac{4096}{225\pi^2}+O(|z|^{-1/2}),
 \qquad \frac{4096}{225\pi^2}>\sigma.
\end{equation}
This uniformity on the whole cut plane, rather than merely on two boundary
rays, follows because the connection formulae consist of convergent series in
$1/z$ and their products with $\Log(-z)$.  Once $|z|$ is large, the series
remainders are uniformly geometric throughout $\D$, and
$|\Log(-z)|\le\log|z|+\pi$.  The leading constant is checked in
Appendix~\ref{app:infinity}.

Fix $n>16$ and $m=2n-1$.  Define the moving contour
\begin{align*}
 \Hankel_{n,-}&=\{r-i/n:r\ge\rho\}, &&r:\infty\to\rho,\\
 \Hankel_{n,0}&=\{\rho-n^{-1}e^{i\theta}:-\pi/2\le\theta\le\pi/2\},
 &&\theta:\pi/2\to-\pi/2,\\
 \Hankel_{n,+}&=\{r+i/n:r\ge\rho\}, &&r:\rho\to\infty,
\end{align*}
and $\Hankel_n=\Hankel_{n,-}\cup\Hankel_{n,0}\cup\Hankel_{n,+}$.
The semicircle passes to the left of $\rho$.  By
Lemma~\ref{lem:zero-free},
$F(z)^{-m}z^{-n-1}$ is holomorphic in the deformation region apart from the
origin.  Moreover, \eqref{eq:F-infinity} gives
\[
 F(z)^{-m}z^{-n-1}=O_n(|z|^{-3/2}),
\]
uniformly on the required large arcs.  More explicitly, for $R$ large we
perform the deformation in $\D\cap\{|z|<R\}$ and close the two truncated
horizontal pieces by the outer circular arc with the induced orientation.
The deformation stays in $\D$ and crosses neither the origin nor a zero of
$F$.  Since the closing arc has length $O(R)$, its integral is
$O_n(R^{-1/2})$.  Letting $R\to\infty$ therefore gives the exact identity
\begin{equation}\label{eq:moving-Hankel}
 S(n)=-\frac1m\frac1{2\pi i}\int_{\Hankel_n}
 F(z)^{-m}z^{-n-1}\dd z.
\end{equation}
All points of this contour remain inside the holomorphy domain of $F$.

Scale
\begin{equation}\label{eq:Hankel-scaling}
 z=\rho\left(1+\frac tn\right),
 \qquad a:=\rho^{-1}=16.
\end{equation}
The image is the fixed contour $\Hankel_a$ consisting of the lower ray
$x-ia$ directed from infinity to zero, the left semicircle
$-ae^{i\theta}$ directed from $-ia$ to $ia$, and the upper ray $x+ia$
directed from zero to infinity.  Thus
\begin{equation}\label{eq:scaled-Hankel}
 S(n)=-\frac{\rho^{-n}}{mn}\frac1{2\pi i}
 \int_{\Hankel_a}
 F\!\left(\rho\left(1+\frac tn\right)\right)^{-m}
 \left(1+\frac tn\right)^{-n-1}\dd t.
\end{equation}

\subsection{Hankel-contour evaluation of $S(n)$}

Use the functions $P,q_0$ from Lemma~\ref{lem:slit-uniform} in the exact
decomposition \eqref{eq:exact-Pq}.
On $\Hankel_a$, $u=-t/n$ and
\begin{equation}\label{eq:log-scale}
 \Log(-t/n)=\Log(-t)-\log n.
\end{equation}
Let $L_n=(\log n)^2$ and split $\Hankel_a$ into
\begin{align}
 \Hankel_a^{(0)}&=\Hankel_a\cap\{\re t\le L_n\},\notag\\
 \Hankel_a^{(1)}&=\Hankel_a\cap\{L_n<\re t\le\delta n\},\notag\\
 \Hankel_a^{(2)}&=\Hankel_a\cap\{\re t>\delta n\}.
 \label{eq:Hankel-split}
\end{align}

We first record the elementary complex estimate used for the negative integer
power.

\begin{lemma}[Finite-geometric estimate]\label{lem:negative-power}
Fix $0<\xi_0<1/4$.  If $m\ge1$, $\xi\in\C$, and $m|\xi|\le\xi_0$, then
\[
 (1+\xi)^{-m}=1-m\xi+E_m(\xi),
 \qquad |E_m(\xi)|\le C_{\xi_0}m^2|\xi|^2.
\]
\end{lemma}

\begin{proof}
The finite geometric identity gives
\[
 (1+\xi)^{-m}-1=-\xi\sum_{k=1}^m(1+\xi)^{-k}.
\]
After adding $m\xi$ and using
$1-(1+\xi)^{-k}=\xi\sum_{j=1}^k(1+\xi)^{-j}$, we have
\[
 |(1+\xi)^{-m}-1+m\xi|
 \le e^{2\xi_0}|\xi|^2\sum_{k=1}^mk
 \le e^{2\xi_0}m^2|\xi|^2,
\]
because $|(1+\xi)^{-j}|\le(1-|\xi|)^{-j}\le e^{2j|\xi|}$.
\end{proof}

\begin{proposition}[Local Hankel contribution]\label{prop:local-Hankel}
Let $\mathcal L_n$ denote the part of the scaled Hankel integral
\eqref{eq:scaled-Hankel} over $\Hankel_a^{(0)}$; explicitly,
\[
 \mathcal L_n:=-\frac{\rho^{-n}}{mn}\frac1{2\pi i}
 \int_{\Hankel_a^{(0)}}
 F\!\left(\rho\left(1+\frac tn\right)\right)^{-m}
 \left(1+\frac tn\right)^{-n-1}\dd t.
\]
Then
\[
 \mathcal L_n=-\frac{24B}{A^5}\,\sigma^{-n}n^{-5}
 +o(\sigma^{-n}n^{-5}).
\]
\end{proposition}

\begin{proof}
On $\Hankel_a^{(0)}$ put
\[
 P_n(t):=P(-t/n),
 \qquad h_n(t):=q_0(-t/n)(\Log(-t)-\log n).
\]
Then
\[
 F\!\left(\rho\left(1+\frac tn\right)\right)=P_n(t)+h_n(t).
\]
For large $n$, $|P_n(t)|\ge F_0/2$ uniformly on the local contour.  Since the
fixed contour stays a distance $a$ from the logarithmic cut,
\[
 |\Log(-t)|\le C_{\mathrm{log}}(1+\log(1+|t|)),
\]
and therefore
\begin{equation}\label{eq:small-hn}
 m\left|\frac{h_n(t)}{P_n(t)}\right|
 \le C_{\mathrm{small}}n^{-3}(\log n)^9\longrightarrow0.
\end{equation}
Writing $C_E:=C_{\xi_0}$ and applying Lemma~\ref{lem:negative-power},
\begin{equation}\label{eq:negative-power-expansion}
 (P_n+h_n)^{-m}=P_n^{-m}-mP_n^{-m-1}h_n+E_n,
 \qquad
 |E_n|\le C_E|P_n|^{-m}m^2\left|\frac{h_n}{P_n}\right|^2.
\end{equation}
Moreover,
\begin{equation}\label{eq:hn-leading}
 h_n(t)=\frac{Bt^4}{n^4}(\Log(-t)-\log n)+r_n(t),
\end{equation}
with
\begin{equation}\label{eq:rn-bound}
 |r_n(t)|\le C_r\frac{|t|^5}{n^5}
   (\log n+1+\log(1+|t|)).
\end{equation}

Define the holomorphic phase near zero by
\begin{equation}\label{eq:phase}
 \lambda(w):=\Log(1+w)+2\Log\!\left(\frac{P(-w)}{F_0}\right),
\end{equation}
where the second logarithm is the branch vanishing at zero.  From
\eqref{eq:P-data},
\begin{equation}\label{eq:phase-expansion}
 \lambda(0)=0,\qquad \lambda'(0)=A,\qquad
 \lambda(w)=Aw+O(w^2).
\end{equation}
The exact algebraic identity
\begin{equation}\label{eq:phase-identity}
 P_n(t)^{-2n}\left(1+\frac tn\right)^{-n-1}
 =F_0^{-2n}\frac{e^{-n\lambda(t/n)}}{1+t/n}
\end{equation}
is the source of the endpoint Laplace factor.  On the horizontal pieces
$t=x\pm ia$, $0\le x\le L_n$, shrinking the local neighborhood if necessary,
\begin{equation}\label{eq:phase-bound}
 \re\lambda((x\pm ia)/n)\ge\frac{Ax}{2n}-\frac {C_\lambda}{n^2},
 \qquad
 |e^{-n\lambda((x\pm ia)/n)}|\le C_{\exp}e^{-Ax/2}.
\end{equation}

More explicitly, choose $C_{\rm quad}>0$ so that
$|\lambda(w)-Aw|\le C_{\rm quad}|w|^2$ near zero and put
$C_\lambda=a^2C_{\rm quad}$.  Choose $N$ so that
$C_{\rm quad}L_n/n\le A/2$ and all preceding estimates hold for $n\ge N$.
  Then the first inequality gives
\[
 |e^{-n\lambda((x\pm ia)/n)}|
 \le e^{C_\lambda/n}e^{-Ax/2}
 \le e^{C_\lambda/N}e^{-Ax/2},
\]
On the fixed left semicircle, the relation
$n\lambda(t/n)=At+O(n^{-1})$ holds uniformly; hence the exponential factor is
bounded there by some $M_0>0$.  We may therefore take
\[
 C_{\exp}:=\max\{e^{C_\lambda/N},M_0\}.
\]
Thus $C_\lambda$ and $C_{\exp}$ are fixed constants independent of $n,x,t$,
but they are not being identified with one another.

We now estimate separately every term on the right-hand side of
\eqref{eq:negative-power-expansion}.  Put
\[
 x_+(t):=\max\{\re t,0\}.
\]
Since $L_n/n\to0$, uniformly on $\Hankel_a^{(0)}$ we have
\[
 \left|1+\frac tn\right|^{-1}\le2,
 \qquad
 \left|\frac{e^{-n\lambda(t/n)}}{1+t/n}\right|
 \le C_{\mathrm{maj}}e^{-Ax_+(t)/2},
 \qquad C_{\mathrm{maj}}:=2C_{\exp}.
\]
The second estimate combines \eqref{eq:phase-bound} on the horizontal
pieces with boundedness on the fixed semicircle.  Equations
\eqref{eq:hn-leading}--\eqref{eq:rn-bound} also give the uniform bound
\[
 |h_n(t)|\le C_h\frac{|t|^4}{n^4}
 \bigl(\log n+1+\log(1+|t|)\bigr).
\]
All constants $C_{\mathrm{log}},C_{\mathrm{small}},C_E,C_r,C_\lambda,
C_{\exp},C_{\mathrm{maj}},C_h$ are independent of $n$ and of the point $t$ on
the indicated contour.  We now fix the constants used at successive stages
of the two remainder estimates.  For $n\ge e$, put
$\ell_n(t)=\log n+1+\log(1+|t|)$.  Then
\[
 \ell_n(t)\le\log n\,[2+\log(1+|t|)]
 \le2\log n\,[1+\log(2+|t|)].
\]
On the full fixed contour $\Hankel_a$, define
\begin{align*}
 M_r&=\int_{\Hankel_a}|t|^5[2+\log(1+|t|)]
               e^{-Ax_+(t)/2}|\dd t|,\\
 M_2&=\int_{\Hankel_a}(1+|t|)^8[1+\log(2+|t|)]^2
               e^{-Ax_+(t)/2}|\dd t|.
\end{align*}
Both constants are finite by exponential decay on the horizontal rays and
compactness of the semicircle.  Fix
\begin{align*}
 K_{r,1}&=\frac{C_{\mathrm{maj}}}{2\pi},&
 K_{r,2}&=C_rK_{r,1},& K_{r,3}&=M_rK_{r,2},\\
 K_{2,1}&=\frac{C_E}{2\pi},&
 K_{2,2}&=\frac{4C_{\mathrm{maj}}}{F_0}K_{2,1},&
 K_{2,3}&=4C_h^2K_{2,2},\\
 K_{2,4}&=M_2K_{2,3}.&&&
\end{align*}
The factor in $K_{2,2}$ uses $m/n\le2$ and $|P_n|^{-1}\le2/F_0$;
the factor $4$ in $K_{2,3}$ comes from the preceding bound on $\ell_n^2$.
Thus every $K_{r,j}$ and $K_{2,j}$ below is fixed before $n$ varies.
The constants $K_0$ and $K_{\mathrm{an}}$ are likewise independent of
$n,t$.

\paragraph{The purely holomorphic term $P_n^{-m}$.}
Let $\mathcal J_{n,0}$ denote the contribution of $P_n^{-m}$ to the local
part of \eqref{eq:scaled-Hankel}.  Because $m=2n-1$,
\begin{align*}
 P_n(t)^{-m}\left(1+\frac tn\right)^{-n-1}
 &=P_n(t)P_n(t)^{-2n}
     \left(1+\frac tn\right)^{-n-1}\\
 &=F_0^{-2n}P_n(t)
   \frac{e^{-n\lambda(t/n)}}{1+t/n}.
\end{align*}
Close $\Hankel_a^{(0)}$ by the vertical segment
\[
 \mathcal V_n=\{L_n+iy:-a\le y\le a\}.
\]
The resulting closed region contains neither $t=-n$ nor a zero of
$P(-t/n)$, so the displayed integrand is holomorphic there.  The expansion
$\lambda(w)=Aw+O(w^2)$ gives, uniformly on $\mathcal V_n$,
\[
 \re\lambda((L_n+iy)/n)\ge\frac{AL_n}{2n}
\]
for all sufficiently large $n$.  Since $P_n$ is uniformly bounded and
$\mathcal V_n$ has fixed length $2a$, Cauchy's theorem and
$\rho^{-n}F_0^{-2n}=\sigma^{-n}$ yield
\begin{align*}
 |\mathcal J_{n,0}|
 &\le K_0\frac{\rho^{-n}F_0^{-2n}}{mn}
       e^{-AL_n/2}\\
 &\le K_0\sigma^{-n}n^{-2}e^{-AL_n/2}
  =o(\sigma^{-n}n^{-5}),
\end{align*}
because $L_n=(\log n)^2$.  Thus the term $P_n^{-m}$ is exponentially
smaller than the desired main scale.

\paragraph{The linear term $-mP_n^{-m-1}h_n$.}
The two minus signs---one in \eqref{eq:scaled-Hankel} and one in
\eqref{eq:negative-power-expansion}---cancel.  Since $m+1=2n$, its local
contribution is
\[
 \frac{\rho^{-n}}n\frac1{2\pi i}
 \int_{\Hankel_a^{(0)}}
 P_n(t)^{-2n}h_n(t)
 \left(1+\frac tn\right)^{-n-1}\dd t.
\]
Substituting the leading part of \eqref{eq:hn-leading} and using
\eqref{eq:phase-identity} shows that the part containing $B$ is
\begin{equation}\label{eq:B-integral}
 B\sigma^{-n}n^{-5}\frac1{2\pi i}
 \int_{\Hankel_a^{(0)}}\frac{e^{-n\lambda(t/n)}}{1+t/n}
 t^4(\Log(-t)-\log n)\dd t.
\end{equation}
Let $\mathcal J_{n,r}$ be the part obtained by replacing $h_n$ with $r_n$.
Using \eqref{eq:phase-identity}, \eqref{eq:rn-bound}, and the preceding
exponential majorant, we obtain
\begin{align*}
 |\mathcal J_{n,r}|
 &\le K_{r,1}\frac{\rho^{-n}F_0^{-2n}}n
 \int_{\Hankel_a^{(0)}}
 |r_n(t)|e^{-Ax_+(t)/2}|\dd t|\\
 &\le K_{r,2}\sigma^{-n}n^{-6}
 \int_{\Hankel_a^{(0)}} |t|^5
 \bigl(\log n+1+\log(1+|t|)\bigr)
 e^{-Ax_+(t)/2}|\dd t|\\
 &\le K_{r,3}\sigma^{-n}n^{-6}\log n
  =o(\sigma^{-n}n^{-5}).
\end{align*}
In the last step, after factoring out $\log n$, the two horizontal pieces are
dominated, upon writing $t=x\pm ia$, by an integrable multiple of
\[
 (1+x)^5\bigl(1+\log(2+x)\bigr)e^{-Ax/2},
\]
and the fixed semicircle contributes only a bounded amount.

\paragraph{The quadratic error $E_n$.}
Let $\mathcal J_{n,2}$ be the contribution of $E_n$ in
\eqref{eq:negative-power-expansion}.  Since
\[
 |E_n(t)|\le C_E m^2|P_n(t)|^{-m-2}|h_n(t)|^2
 =C_E m^2|P_n(t)|^{-2n-1}|h_n(t)|^2,
\]
the outer factor in \eqref{eq:scaled-Hankel}, the bound
$|P_n|^{-1}\le2/F_0$, and $m/n\le2$ give
\begin{align*}
 |\mathcal J_{n,2}|
 &\le K_{2,1}\rho^{-n}\frac mn
 \int_{\Hankel_a^{(0)}}
 |P_n(t)|^{-2n-1}|h_n(t)|^2
 \left|1+\frac tn\right|^{-n-1}|\dd t|\\
 &\le K_{2,2}\sigma^{-n}
 \int_{\Hankel_a^{(0)}}|h_n(t)|^2
 e^{-Ax_+(t)/2}|\dd t|\\
 &\le K_{2,3}\sigma^{-n}n^{-8}(\log n)^2
 \int_{\Hankel_a^{(0)}}
 (1+|t|)^8\bigl(1+\log(2+|t|)\bigr)^2
 e^{-Ax_+(t)/2}|\dd t|\\
 &\le K_{2,4}\sigma^{-n}n^{-8}(\log n)^2
  =o(\sigma^{-n}n^{-5}).
\end{align*}
This displays explicitly that the factor $1/(mn)$ outside the integral reduces
the factor $m^2$ in $E_n$ to $m/n=O(1)$, whereas $h_n^2$ supplies the factor
$n^{-8}$.

\paragraph{The analytic $-\log n$ part of the $B$-term.}
The function
\[
 t\longmapsto \frac{e^{-n\lambda(t/n)}}{1+t/n}t^4
\]
is holomorphic in the same closed region used above.  Closing by
$\mathcal V_n$ therefore gives
\[
 \left|\int_{\Hankel_a^{(0)}}
 \frac{e^{-n\lambda(t/n)}}{1+t/n}t^4\dd t\right|
 \le K_{\mathrm{an}}(1+L_n)^4e^{-AL_n/2}.
\]
Consequently the part of \eqref{eq:B-integral} multiplied by $-\log n$ is,
including its prefactor, at most
\[
 K_{\mathrm{an}}\sigma^{-n}n^{-5}\log n\,(1+L_n)^4e^{-AL_n/2}
 =o(\sigma^{-n}n^{-5}).
\]
The factor $(1+L_n)^4$, which comes from $t^4$ on the closing segment, is
still dominated by the Gaussian-in-$\log n$ decay
$e^{-A(\log n)^2/2}$.

Combining the four estimates above, the local contribution is
\begin{equation}\label{eq:local-contribution}
 B\sigma^{-n}n^{-5}I_n+o(\sigma^{-n}n^{-5}),
\end{equation}
where
\begin{equation}\label{eq:In}
 I_n=\frac1{2\pi i}\int_{\Hankel_a^{(0)}}
 \frac{e^{-n\lambda(t/n)}}{1+t/n}t^4\Log(-t)\dd t.
\end{equation}

To put the limiting argument on a fixed contour, for all sufficiently large
$n$ define on $\Hankel_a$
\[
 J_n(t)=
 \begin{cases}
 \displaystyle\frac{e^{-n\lambda(t/n)}}{1+t/n}t^4\Log(-t),
   &x_+(t)\le L_n,\\[5pt]
 0,&x_+(t)>L_n.
 \end{cases}
\]
Then $I_n=(2\pi i)^{-1}\int_{\Hankel_a}J_n(t)\,\dd t$.  For every fixed
$t\in\Hankel_a$, one has $J_n(t)\to e^{-At}t^4\Log(-t)$ because
$L_n\to\infty$.  Equations \eqref{eq:phase-bound} and the bound following it
give the fixed integrable majorant
\[
 C(1+|t|)^4\bigl(1+\log(1+|t|)\bigr)e^{-Ax_+(t)/2}
\]
on the two horizontal rays; on the semicircle the convergence is uniform.
Dominated convergence on the fixed contour therefore gives
\begin{equation}\label{eq:In-limit}
 I_n\longrightarrow I_a(A):=\frac1{2\pi i}\int_{\Hankel_a}
 e^{-At}t^4\Log(-t)\dd t.
\end{equation}
To evaluate this integral, truncate at $\re t=R$, close vertically, and then
press the contour onto the positive real axis.  The connecting integral is
$O(R^4(1+\log R)e^{-AR})$.  Since
\[
 \Log(-(x-i0))=\log x+i\pi,\qquad
 \Log(-(x+i0))=\log x-i\pi,
\]
and the lower side is traversed from infinity to zero,
\begin{equation}\label{eq:fixed-Hankel-value}
 I_a(A)=-\int_0^\infty e^{-Ax}x^4\dd x
 =-\frac{\Gamma(5)}{A^5}=-\frac{24}{A^5}.
\end{equation}
The answer is independent of $a>0$.  Substitution into
\eqref{eq:local-contribution} proves the asserted estimate for $\mathcal L_n$.
\end{proof}

\begin{proposition}[Intermediate and far Hankel contributions]
\label{prop:intermediate-far-Hankel}
For a sufficiently small fixed $\delta>0$, the parts of the
scaled Hankel integral \eqref{eq:scaled-Hankel} over
$\Hankel_a^{(1)}$ and $\Hankel_a^{(2)}$ are both
$o\bigl(\sigma^{-n}n^{-5}\bigr)$.
\end{proposition}

\begin{proof}

On the intermediate rays write
\[
 z_{n,\pm}(x)=\rho\left(1+\frac xn\right)\pm\frac in,
 \qquad L_n\le x\le\delta n.
\]
Put $y=x/n$ and $h=y\pm i/(\rho n)$, so that
$z_{n,\pm}(x)=\rho(1+h)$ and $u=-h$.  Since $x\ge L_n\to\infty$,
we have $|h|\asymp y$ and
$1+|\Log(-h)|=O(1+\log(1/y))$ uniformly on these rays.
Fix a sufficiently small slit disk as in Lemma~\ref{lem:slit-uniform}.
Using \eqref{eq:exact-Pq} and $P'(0)=-\rho F'(\rho)$, we obtain
\[
 F(z_{n,\pm}(x))=F_0+\rho F'(\rho)h+O(E_{n,x}),
 \qquad E_{n,x}:=y^2+n^{-2}+y^4(1+\log(1/y)).
\]
Here and below the implied constants are independent of $n,x$ and of
subsequent decreases of $\delta$.  Since $|h|^2=O(E_{n,x})$ and
$h,E_{n,x}$ are uniformly bounded, multiplication gives
\[
 \begin{aligned}
 \Phi(z_{n,\pm}(x))
 &=\rho(1+h)\bigl(F_0+\rho F'(\rho)h+O(E_{n,x})\bigr)^2\\
 &=\sigma\left[1+
 \left(1+\frac{2\rho F'(\rho)}{F_0}\right)h+O(E_{n,x})\right]
 =\sigma[1+Ah+O(E_{n,x})].
 \end{aligned}
\]
Taking real parts, using $A\in\mathbb R$ and $\re h=x/n$, yields
\begin{equation}\label{eq:intermediate-expansion}
 \re\Phi(z_{n,\pm}(x))
 =\sigma\left[1+A\frac xn
 +O\!\left(\frac{x^2}{n^2}+\frac1{n^2}
 +\frac{x^4}{n^4}\left(1+\log\frac nx\right)\right)\right].
\end{equation}
To bound the error relative to the positive linear term, observe that
\[
 \frac{E_{n,x}}{y}
 =y+\frac1{nx}+y^3(1+\log(1/y))
 \le\delta+\frac1{nL_n}
 +\sup_{0<s\le\delta}s^3(1+\log(1/s)).
\]
The first and third terms tend to zero as $\delta\downarrow0$;
the second tends to zero as $n\to\infty$.  We may therefore first fix
$\delta>0$ sufficiently small and then take $n$ sufficiently large so
that the absolute error inside the brackets in
\eqref{eq:intermediate-expansion} is at most $3Ay/4$.
Finally,
\[
 |z_{n,\pm}(x)|\,|F(z_{n,\pm}(x))|^2
 =|\Phi(z_{n,\pm}(x))|\ge\re\Phi(z_{n,\pm}(x)),
\]
which gives
\begin{equation}\label{eq:intermediate-lower}
 |z_{n,\pm}(x)|\,|F(z_{n,\pm}(x))|^2
 \ge\sigma\left(1+c\frac xn\right),
 \qquad c=A/4.
\end{equation}
In making the preceding choice of $\delta$, decrease it once more, if
necessary, so that $c\delta\le1$.  The identity
\begin{equation}\label{eq:integrand-modulus-identity}
 |F(z)|^{-m}|z|^{-n-1}
 =\frac{|F(z)|/|z|}{\bigl(|z|\,|F(z)|^2\bigr)^n}
\end{equation}
and local boundedness of $|F(z)|/|z|$ now give the required exponential bound
uniformly: for $L_n\le x\le\delta n$, put $s=cx/n\in[0,1]$ and use
$\log(1+s)\ge s/2$ to obtain
\[
 (1+cx/n)^{-n}=e^{-n\log(1+s)}\le e^{-cx/2}.
\]
It follows that the intermediate
contribution is bounded by
\begin{equation}\label{eq:intermediate-small}
 K_{\mathrm{mid}}\sigma^{-n}n^{-2}
 \int_{L_n}^{\delta n}e^{-cx/2}\dd x
 =o(\sigma^{-n}n^{-5}).
\end{equation}
Indeed, the integral is at most $(2/c)e^{-cL_n/2}$, and hence the ratio of
the displayed bound to $\sigma^{-n}n^{-5}$ is at most a fixed multiple of
$n^3e^{-c(\log n)^2/2}$, which tends to zero.

Put $r_\delta=\rho(1+\delta)$ and
\[
 L_\infty:=\frac{4096}{225\pi^2}>\sigma.
\]
By the uniform estimate \eqref{eq:Phi-infinity}, one may choose a fixed
$M>r_\delta$ so large that
\begin{equation}\label{eq:far-tail-lower}
 |z|\,|F(z)|^2\ge q_\infty:=\frac{L_\infty+\sigma}{2}>\sigma
 \qquad(|z|\ge M, z\in\D).
\end{equation}
Increasing $M$ if necessary, the condition $r\ge M$ implies $|r\pm i/n|\ge M$,
so \eqref{eq:far-tail-lower} applies to the two moving rays.

It remains to obtain a bound on the fixed compact interval
$I_{\delta,M}:=[r_\delta,M]$.  Lemma~\ref{lem:cut-neighborhood} constructs holomorphic
branches $\widetilde F_\pm$ on complex neighborhoods $U_\pm\supset I_{\delta,M}$
such that
\[
 \widetilde F_\pm(z)=F(z)\quad
 (z\in U_\pm,\ \pm\im z>0),
 \qquad
 \widetilde F_\pm(r)=F_\pm(r)\quad(r\in I_{\delta,M}).
\]
Since $I_{\delta,M}$ is compact, there is an $\eta>0$ such that the closed
one-sided strips
\[
 \mathcal K_\pm:=
 \{r+iy:r\in I_{\delta,M},\ 0\le\pm y\le\eta\}
\]
are contained in $U_\pm$.  Compactness and holomorphy give finite constants
\[
 D_{\delta,M}:=
 \max_{\epsilon\in\{+,-\}}\sup_{z\in\mathcal K_\epsilon}
 |\widetilde F_\epsilon'(z)|,
 \qquad
 H_{\delta,M}:=
 \max_{\epsilon\in\{+,-\}}\sup_{z\in\mathcal K_\epsilon}
 |\widetilde F_\epsilon(z)|.
\]
For $n\ge\eta^{-1}$, integration along the vertical segment from $r$ to
$r\pm i/n$ therefore yields
\begin{equation}\label{eq:boundary-uniform-convergence}
 \sup_{r\in I_{\delta,M}}
 |F(r\pm i/n)-F_\pm(r)|\le\frac{D_{\delta,M}}n.
\end{equation}

The continuous boundary functions
\[
 G_\pm(r):=r|F_\pm(r)|^2
\]
satisfy $G_\pm(r)>\sigma$ by Theorem~\ref{thm:cut-separation}.  Hence their
minimum on the compact interval has a positive margin:
\begin{equation}\label{eq:compact-boundary-margin}
 m_{\delta,M}:=
 \min_{\epsilon\in\{+,-\}}\min_{r\in I_{\delta,M}}G_\epsilon(r)>\sigma,
 \qquad d_{\delta,M}:=m_{\delta,M}-\sigma>0.
\end{equation}
Moreover,
\[
 \bigl||r\pm i/n|-r\bigr|
 =\frac{n^{-2}}{|r\pm i/n|+r}
 \le\frac1{2r_\delta n^2}.
\]
Combining this estimate with \eqref{eq:boundary-uniform-convergence} and the
bound $H_{\delta,M}$ shows explicitly that
\begin{align*}
 &\sup_{r\in I_{\delta,M}}
 \left|
 |r\pm i/n|\,|F(r\pm i/n)|^2-r|F_\pm(r)|^2
 \right|\\
 &\qquad\le
 \frac{H_{\delta,M}^2}{2r_\delta n^2}
 +\frac{2M H_{\delta,M}D_{\delta,M}}n
 \longrightarrow0.
\end{align*}
Here $H_{\delta,M}$ may be enlarged once to bound both the values on the
strips and their boundary restrictions.  Thus, for all sufficiently large
$n$, the last supremum is smaller than $d_{\delta,M}/2$, and consequently
\[
 |r\pm i/n|\,|F(r\pm i/n)|^2
 \ge q_{\delta,M}^{\mathrm{comp}}
 :=\sigma+\frac{d_{\delta,M}}2>\sigma
 \qquad(r\in I_{\delta,M}).
\]
Taking
\[
 q_\delta:=\min\{q_\infty,q_{\delta,M}^{\mathrm{comp}}\}>\sigma
\]
combines the compact and infinite ranges and gives, independently of all
sufficiently large $n$,
\begin{equation}\label{eq:far-lower}
 |r\pm i/n|\,|F(r\pm i/n)|^2\ge q_\delta
 \qquad(r\ge r_\delta).
\end{equation}

We also need an integrable upper majorant for the numerator in
\eqref{eq:integrand-modulus-identity}.  On $I_{\delta,M}$, the bound just
obtained from the compact strips and $|r\pm i/n|\ge r_\delta$ gives
\[
 \frac{|F(r\pm i/n)|}{|r\pm i/n|}
 \le\frac{H_{\delta,M}}{r_\delta}
 \le\frac{H_{\delta,M}}{r_\delta}(1+M)^{3/2}(1+r)^{-3/2}.
\]
For $r\ge M$, the infinity expansion \eqref{eq:F-infinity} gives instead,
uniformly on the moving rays,
\[
 \frac{|F(r\pm i/n)|}{|r\pm i/n|}
 \le K_\infty r^{-3/2}
 \le K_\infty(1+M^{-1})^{3/2}(1+r)^{-3/2}.
\]
After taking the larger of these two fixed constants, there is therefore a
$K_\delta>0$, independent of all sufficiently large $n$, such that
\begin{equation}\label{eq:far-numerator-majorant}
 \frac{|F(r\pm i/n)|}{|r\pm i/n|}
 \le K_\delta(1+r)^{-3/2}\qquad(r\ge r_\delta).
\end{equation}
Equations \eqref{eq:integrand-modulus-identity}, \eqref{eq:far-lower}, and
\eqref{eq:far-numerator-majorant} now show that the far contribution is at most
\begin{equation}\label{eq:far-small}
 \frac{K_{\mathrm{far}}q_\delta^{-n}}n
 \int_{r_\delta}^\infty(1+r)^{-3/2}\dd r
 =o(\sigma^{-n}n^{-5}).
\end{equation}
Here the integral is finite and independent of $n$, while
$n^4(\sigma/q_\delta)^n\to0$ because $q_\delta>\sigma$.
This proves the proposition.
\end{proof}

\begin{theorem}[Coefficient asymptotics from the Hankel contour]
\label{thm:S-coeff-Hankel}
As $n\to\infty$,
\[
 S(n)\sim-\frac{24B}{A^5}\,\sigma^{-n}n^{-5}
 =\frac{24}{\pi A^5}\,\sigma^{-n}n^{-5}
 =\frac{24(165\pi-512)^5}{\pi(1280-405\pi)^5}\,
   \sigma^{-n}n^{-5}.
\]
\end{theorem}

\begin{proof}
The decomposition \eqref{eq:Hankel-split} of the exact contour identity
\eqref{eq:scaled-Hankel} is exhaustive.  Proposition~\ref{prop:local-Hankel}
gives the displayed main term on $\Hankel_a^{(0)}$, while
Proposition~\ref{prop:intermediate-far-Hankel} shows that the contributions
from $\Hankel_a^{(1)}$ and $\Hankel_a^{(2)}$ are
$o(\sigma^{-n}n^{-5})$.  Finally $B=-1/\pi$ by \eqref{eq:exact-Pq}.
This is precisely Theorem~\ref{thm:S-coeff-main}.
\end{proof}

In particular, there is $C_0>0$ such
that
\begin{equation}\label{eq:coefficient-bound}
 |S(n)|\le C_0\sigma^{-n}n^{-5}
\end{equation}
for all sufficiently large $n$.

\section{Global inversion and continuation across nonprincipal boundary points}
\label{sec:global-inversion}

The estimate \eqref{eq:coefficient-bound} implies that the series for $S$ and
$S'$ converge absolutely and uniformly on $|y|\le\sigma$, because
\[
 |S(n)|\sigma^n=O(n^{-5}),
 \qquad n|S(n)|\sigma^{n-1}=O(n^{-4}).
\]
Thus
\begin{equation}\label{eq:boundary-regularity}
 S,S'\text{ extend continuously to }|y|\le\sigma,
 \qquad \sup_{|y|\le\sigma}|S'(y)|<\infty.
\end{equation}
This conclusion has been obtained before any use of $\Delta$-analyticity.

Consider
\[
 \Phi^{-1}(\B_\sigma)=\{z\in\D:|\Phi(z)|<\sigma\}
\]
and define
\begin{equation}\label{eq:Omega-def}
 \Omega:=\text{the connected component containing $0$ of }
             \Phi^{-1}(\B_\sigma).
\end{equation}

\begin{definition}
A continuous map $f:X\to Y$ is \emph{proper} if $f^{-1}(K)$ is compact in
$X$ for every compact $K\subset Y$.
\end{definition}

\begin{proposition}[Properness]\label{prop:properness}
The map $\Phi:\Omega\to\B_\sigma$ is proper.
\end{proposition}

\begin{proof}
Let $K\Subset\B_\sigma$.  Choose $\epsilon>0$ so that
$|y|\le\sigma-\epsilon$ on $K$, and take any sequence
$z_j\in\Phi^{-1}(K)\cap\Omega$.

The sequence cannot be unbounded.  Otherwise \eqref{eq:Phi-infinity} would
give $|\Phi(z_j)|\ge q_\infty>\sigma$ along a subsequence.  It cannot approach
an interior point $r>\rho$ of the cut.  Indeed, after passing to a subsequence
lying entirely in the upper or lower half-plane,
Lemma~\ref{lem:cut-neighborhood} gives $F(z_j)\to F_\pm(r)$.  Consequently,
Theorem~\ref{thm:cut-separation} gives
\[
 |\Phi(z_j)|=|z_j|\,|F(z_j)|^2
 \longrightarrow r|F_\pm(r)|^2>\sigma.
\]
Nor can it approach the endpoint $\rho$, because Lemma~\ref{lem:slit-uniform} gives
$\Phi(z)\to\rho F_0^2=\sigma$.  Each possibility contradicts
$|\Phi(z_j)|\le\sigma-\epsilon$.

Hence the sequence remains in a compact subset of $\D$ and has a convergent
subsequence $z_{j_k}\to z_*\in\D$.  Continuity gives $\Phi(z_*)\in K$.  Since
$|\Phi(z_*)|<\sigma$, a small connected disk $V$ about $z_*$ lies in
$\Phi^{-1}(\B_\sigma)$.  For large $k$, $z_{j_k}\in V\cap\Omega$; maximality
of the connected component then gives $V\subset\Omega$.  Thus $z_*\in\Omega$,
proving compactness of the preimage.
\end{proof}

We shall also use the elementary fact that a proper continuous map between the
metric spaces in question is closed.  Indeed, if $C\subset\Omega$ is closed and
$y_j=\Phi(z_j)\in\Phi(C)$ converges to $y\in\B_\sigma$, then
$K=\{y\}\cup\{y_j:j\ge1\}$ is compact.  Properness gives a convergent
subsequence of $(z_j)$ with limit $z\in C$, and $\Phi(z)=y$.

\begin{theorem}[Biholomorphic inversion on the convergence disk]
\label{thm:biholo}
The map $\Phi:\Omega\to\B_\sigma$ is biholomorphic.  Its inverse
$X=\Phi^{-1}$ satisfies
\begin{equation}\label{eq:inverse-identities}
 X(y)=\frac{y}{S(y)^2},
 \qquad
 S(y)=F(X(y))=F\!\left(\frac{y}{S(y)^2}\right),
 \qquad |y|<\sigma.
\end{equation}
\end{theorem}

\begin{proof}
Since $\Phi'(0)=1$, $\Phi$ is nonconstant.  By the open mapping theorem
(see, e.g., \cite{Ahlfors}), $\Phi(\Omega)$ is open in $\B_\sigma$.
Properness and the preceding closed-map
argument make it closed.  It is nonempty and $\B_\sigma$ is connected; hence
\begin{equation}\label{eq:Phi-surjective}
 \Phi(\Omega)=\B_\sigma.
\end{equation}

The function $S(\Phi(z))$ is holomorphic on $\Omega$ and agrees near zero with
$F(z)$ by the formal identity \eqref{eq:implicit-intro}.  The identity theorem
gives
\begin{equation}\label{eq:S-Phi-F}
 S(\Phi(z))=F(z)\qquad(z\in\Omega).
\end{equation}
For any $z_1,z_2\in\Omega$ satisfying
$\Phi(z_1)=\Phi(z_2)=y$, identity \eqref{eq:S-Phi-F} gives
\[
 F(z_1)=S(y)=F(z_2)\ne0,
\]
where nonvanishing follows from Lemma~\ref{lem:zero-free} because
$z_1,z_2\in\Omega\subset\D$.
Consequently,
\[
 y=z_1F(z_1)^2=z_1S(y)^2,
 \qquad
 y=z_2F(z_2)^2=z_2S(y)^2.
\]
Since $S(y)\ne0$, it follows that $z_1=z_2$.  Hence
$\Phi|_\Omega$ is injective, and its unique preimage of $y$ is
\[
 z=\frac{y}{S(y)^2}.
\]
A univalent holomorphic function has nonzero derivative.  Its local
inverses consequently glue to a holomorphic global inverse, and
\eqref{eq:inverse-identities} follows from \eqref{eq:S-Phi-F}.
\end{proof}

Notice that $X(y)\in\D$ need not satisfy $|X(y)|<\rho$; $F(X(y))$ refers to the
cut-plane continuation, not necessarily direct substitution into the Taylor
series at zero.

\begin{proposition}[The boundary inverse]\label{prop:boundary-inverse}
Let $|y_0|=\sigma$ and $y_0\ne\sigma$.  Then
\begin{equation}\label{eq:boundary-X-limit}
 x_0:=\lim_{\substack{y\to y_0\\|y|<\sigma}}X(y)
\end{equation}
exists, belongs to $\D$, and satisfies
\begin{equation}\label{eq:boundary-X-identities}
 \Phi(x_0)=y_0,\qquad F(x_0)=S(y_0)\ne0,\qquad
 x_0=\frac{y_0}{S(y_0)^2}.
\end{equation}
Moreover, $\Phi'(x_0)\ne0$.
\end{proposition}

\begin{proof}
Take $y_j\to y_0$ with $|y_j|<\sigma$ and put $x_j=X(y_j)$.  The same
separation arguments as in Proposition~\ref{prop:properness} show that $(x_j)$ is
relatively compact in $\D$: unboundedness contradicts
\eqref{eq:Phi-infinity}, approach to an interior point of the cut contradicts
Theorem~\ref{thm:cut-separation}, and approach to $\rho$ would force
$y_0=\sigma$.

If a subsequence tends to $\xi\in\D$, then $\Phi(\xi)=y_0$.  By
\eqref{eq:boundary-regularity} and \eqref{eq:S-Phi-F},
\[
 F(\xi)=\lim F(x_j)=\lim S(y_j)=S(y_0)\ne0.
\]
It follows that
$\xi=y_0/S(y_0)^2$.  Every accumulation point is the same, so the full
sequence converges and \eqref{eq:boundary-X-identities} holds.

Differentiate \eqref{eq:S-Phi-F} in $\Omega$ and pass to the boundary:
\[
 F'(x_0)=S'(y_0)\Phi'(x_0).
\]
If $\Phi'(x_0)=0$, then $F'(x_0)=0$.  But
\[
 \Phi'(x)=F(x)^2+2xF(x)F'(x),
\]
so $\Phi'(x_0)=F(x_0)^2\ne0$, a contradiction.
\end{proof}

\begin{theorem}[Continuation through the nonprincipal circle]
\label{thm:nonprincipal-continuation}
For every $y_0$ with $|y_0|=\sigma$ and $y_0\ne\sigma$, $S$ extends
holomorphically to a neighborhood of $y_0$.
\end{theorem}

\begin{proof}
By Proposition~\ref{prop:boundary-inverse} and the holomorphic inverse function theorem,
there are neighborhoods $V\ni x_0$, $W\ni y_0$ and a holomorphic inverse
$\psi:W\to V$ of $\Phi|_V$.  Shrink $W$ so that the boundary convergence of
$X$ puts $X(y)$ in $V$ for $y\in W\cap\B_\sigma$.  Since
$\Phi(X(y))=y=\Phi(\psi(y))$ and $\Phi|_V$ is injective,
$X(y)=\psi(y)$ on this overlap.  Therefore
\[
 \widetilde S(y):=F(\psi(y)),\qquad y\in W,
\]
is the required continuation.
\end{proof}

The coefficient asymptotic gives
$\limsup_{n\to\infty}|S(n)|^{1/n}=\sigma^{-1}$, so the radius of convergence
is exactly $\sigma$.  By Theorem~\ref{thm:nonprincipal-continuation}, every
point of $|y|=\sigma$ other than $\sigma$ is regular.  If $\sigma$ were regular
as well, the identity theorem would glue all local continuations to the
power-series function on a neighborhood of the closed disk
$|y|\le\sigma$.  Compactness would then
give some $\epsilon_0>0$ for which this neighborhood contains
$|y|<\sigma+\epsilon_0$, contradicting the radius of convergence.  Thus
$\sigma$ is the unique singularity on that circle.  The compatible finite disk
construction in Appendix~\ref{app:monodromy}, including
Lemma~\ref{lem:disk-disk}, gives the following uniform form: for every
$\epsilon>0$, the arc
\begin{equation}\label{eq:Kepsilon}
 K_\epsilon=\{y:|y|=\sigma,\ |y-\sigma|\ge\epsilon\}
\end{equation}
has a neighborhood to which $S$ extends holomorphically.

\section{A genuine $\Delta$-analytic continuation of $S$}
\label{sec:delta-continuation}

\subsection{Local coordinates and the logarithmic perturbation}

Near $z=\rho$ and $y=\sigma$ put
\begin{equation}\label{eq:uv-coordinates}
 u=1-\frac z\rho,\qquad v=1-\frac y\sigma.
\end{equation}
For $0<\vartheta<\pi$ and $\delta>0$, define the slit sector
\begin{equation}\label{eq:sector-def}
 \mathcal S(\delta,\vartheta)
 =\{w:0<|w|<\delta,\ |\Arg w|<\pi-\vartheta\}.
\end{equation}
Let
\begin{equation}\label{eq:fg-local}
 f(u):=F(\rho(1-u)),
 \qquad
 g(u):=1-\frac{\Phi(\rho(1-u))}{\sigma}
      =1-(1-u)\frac{f(u)^2}{F_0^2}.
\end{equation}
The inverse problem is exactly $v=g(u)$.

\begin{lemma}[Local expansion of $g$]\label{lem:g-expansion}
For every fixed $0<\vartheta_0<\pi$ and all sufficiently
small $\delta_0>0$, uniformly in $\mathcal S(\delta_0,\vartheta_0)$,
\begin{equation}\label{eq:g-expansion}
 g(u)=g_0(u)+D_4u^4\Log u
 +O\bigl(u^5(1+|\Log u|)\bigr),
\end{equation}
where
\begin{equation}\label{eq:g0-D4}
 g_0(u)=1-(1-u)\frac{P(u)^2}{F_0^2},
 \qquad D_4=-\frac{2B}{F_0},
\end{equation}
and $g_0(0)=0$, $g_0'(0)=A$.
\end{lemma}

\begin{proof}
Write \eqref{eq:F-endpoint} as
$f(u)=P(u)+Bu^4\Log u+E(u)$ with
$E(u)=O(u^5(1+|\Log u|))$.  Squaring gives
\begin{align*}
 f(u)^2={}&P(u)^2+2P(u)Bu^4\Log u+2P(u)E(u)
 +B^2u^8(\Log u)^2\\
 &+2Bu^4\Log u\,E(u)+E(u)^2.
\end{align*}
Since $P(u)=F_0+O(u)$, the second term is
$2F_0Bu^4\Log u+O(u^5|\Log u|)$; all remaining error terms have the stated
order (in particular $u^8(\Log u)^2=O(u^5)$).  Substitution into
\eqref{eq:fg-local}, using $1-u=1+O(u)$, proves
\eqref{eq:g-expansion}--\eqref{eq:g0-D4}.  Finally,
\[
 g_0'(0)=1-\frac{2P'(0)}{F_0}
 =1+\frac{2\rho F'(\rho)}{F_0}=A.
\]
\end{proof}

\subsection{A logarithmically perturbed sectorial inverse}

Fix $0<\vartheta_0<\vartheta_1<\pi/2$, and choose $\delta_0>0$ sufficiently
small so that Lemma~\ref{lem:g-expansion} holds on
$\mathcal S(\delta_0,\vartheta_0)$.

\begin{lemma}[Sectorial inversion]\label{lem:sectorial-inverse}
There exist $0<\delta_1<\delta_0$ and $\epsilon>0$ such that, for every
$v\in\mathcal S(\epsilon,\vartheta_1)$, the equation $g(u)=v$ has exactly one
solution $u=U(v)$ in $\mathcal S(\delta_1,\vartheta_0)$.  The function $U$ is
holomorphic and there is a polynomial $H_4$ of degree at most four such that
\begin{equation}\label{eq:U-expansion}
 U(v)=H_4(v)-\frac{D_4}{A^5}v^4\Log v
 +O\bigl(v^5(1+|\Log v|)\bigr),
\end{equation}
where $H_4(0)=0$ and $H_4'(0)=A^{-1}$.  The estimate is uniform on every fixed
smaller closed sector.
\end{lemma}

\begin{proof}
We keep track of every smallness condition and make one final choice of
$\epsilon$ at the end.

\paragraph{Fixed local data.}

Since $g_0'(0)=A\ne0$, choose $R_u,R_h>0$ so that $g_0$ is
univalent on $D(0,R_u)$ and $D(0,R_h)\subset g_0(D(0,R_u))$.
Its inverse $h$ is holomorphic on $D(0,R_h)$, with

\begin{equation}\label{eq:h-expansion}
 g_0(h(v))=v,\qquad h(v)=\frac vA+O(v^2).
\end{equation}
Decrease $\delta_0$ once, if necessary; this preserves the validity of
Lemma~\ref{lem:g-expansion}.  We may therefore require that
$\overline{D(0,\delta_0)}\subset D(0,R_u)$,
$|g_0'(u)-A|\le A/4$ for $|u|\le\delta_0$, and the expansion in
Lemma~\ref{lem:g-expansion} has a uniform bound
\begin{equation}\label{eq:r-uniform}
 r(u):=g(u)-g_0(u),\qquad
 |r(u)|\le C|u|^4(1+|\Log u|)
 \quad(u\in\mathcal S(\delta_0,\vartheta_0)).
\end{equation}
This is the only shrinking of the outer $u$-sector.

Put
\[
 \vartheta_c:=\frac{\vartheta_0+\vartheta_1}{2},
 \qquad \gamma:=\vartheta_c-\vartheta_0
       =\frac{\vartheta_1-\vartheta_0}{2}>0,
\]
and fix
\begin{equation}\label{eq:q-geometry}
 0<q<\min\left\{\frac12,\sin\gamma\right\}.
\end{equation}
Because $h(v)/(v/A)\to1$ uniformly as $v\to0$, there are
$\epsilon_h>0$ and constants $c_-,c_+>0$ such that, for every
$v\in\mathcal S(\epsilon_h,\vartheta_1)$,
\begin{equation}\label{eq:h-sector-map}
 c_-|v|\le|h(v)|\le c_+|v|,\qquad
 |h(v)|<\frac{\delta_0}{4},\qquad
 |\Arg h(v)-\Arg v|<\gamma.
\end{equation}
In particular, $h(v)\in\mathcal S(\delta_0/4,\vartheta_c)$ and its angular
distance from each boundary ray of the outer sector is at least $\gamma$.

The elementary geometry used below is now uniform.  If $u_0=h(v)$ and
$0<r\le q|u_0|$, then
\begin{equation}\label{eq:relative-disk-in-sector}
 \overline{D(u_0,r)}\subset\mathcal S(\delta_0,\vartheta_0).
\end{equation}
Indeed, points in the relative disk have modulus between
$(1-q)|u_0|$ and $(1+q)|u_0|$, while their arguments differ from
$\Arg u_0$ by at most $\arcsin q<\gamma$.  This concerns the closed disk, not
merely its boundary circle, as required by Rouch\'e's theorem.

\paragraph{Every local solution has the scale $|u|\asymp|v|$.}
From \eqref{eq:g-expansion} and $g_0(u)=Au+O(u^2)$, uniformly in the outer
sector, $g(u)=Au+o(u)$.  Choose $0<\delta_1<\delta_0/4$ so small that
\begin{equation}\label{eq:g-linear-two-sided}
 \frac A2|u|\le|g(u)|\le\frac{3A}{2}|u|
 \qquad(u\in\mathcal S(\delta_1,\vartheta_0)),
\end{equation}
and so that $g_0(D(0,\delta_1))\subset D(0,R_h/2)$.  Consequently every
solution $g(u)=v$ in this local sector satisfies
\begin{equation}\label{eq:any-root-scale}
 \frac{2}{3A}|v|\le|u|\le\frac2A|v|.
\end{equation}
Thus any family of such solutions tends to zero with $v$; this is a
consequence of the equation, not of its position relative to the Rouch\'e
circle.

Choose $\epsilon_{\mathrm{root}}$ so that
\[
 0<\epsilon_{\mathrm{root}}<R_h/2.
\]
On the convex disk $D(0,R_h/2)$ the derivative $h'$ is bounded.  If $g(u)=v$ with
$u\in\mathcal S(\delta_1,\vartheta_0)$ and
$|v|<\epsilon_{\mathrm{root}}$, then
\[
 g_0(u)=v-r(u),\qquad u=h(v-r(u)).
\]
Using \eqref{eq:any-root-scale}, \eqref{eq:r-uniform}, and the boundedness of
$h'$, we obtain a constant $C_2$, independent of $u,v$, such that
\begin{equation}\label{eq:all-roots-near-h}
 |u-h(v)|\le C_2|v|^4(1+|\Log v|).
\end{equation}
Here $1+|\Log u|\asymp1+|\Log v|$ follows from
\eqref{eq:any-root-scale} and the fixed angular bounds.

\paragraph{The Rouch\'e disk.}
Assume temporarily that a disk centered at $u_0=h(v)$ has relative radius at
most $q|u_0|$.  The bounds \eqref{eq:h-sector-map} then imply, throughout that
disk,
\[
 (1-q)c_-|v|\le|u|\le(1+q)c_+|v|,
 \qquad 1+|\Log u|\asymp1+|\Log v|.
\]
Thus \eqref{eq:r-uniform} gives a constant $C_1$, independent of $v$, for
which
\begin{equation}\label{eq:r-on-relative-disk}
 |r(u)|\le C_1|v|^4(1+|\Log v|).
\end{equation}
Now fix, once and for all,
\begin{equation}\label{eq:M-final-choice}
 M>\max\left\{\frac{4C_1}{3A},C_2\right\},
\end{equation}
and put
\begin{equation}\label{eq:Rouche-circle}
 r_v=M|v|^4(1+|\Log v|),\qquad |u-u_0|=r_v.
\end{equation}
Since $|u_0|\ge c_-|v|$, there exists $\epsilon_R>0$ such that, for every
$0<|v|<\epsilon_R$,
\begin{equation}\label{eq:epsilon-R-condition}
 \frac{r_v}{|u_0|}
 \le\frac{M}{c_-}|v|^3(1+|\Log v|)<q.
\end{equation}
The limit $x^3(1+|\log x|)\to0$ proves the existence of this single, uniform
$\epsilon_R$.  Equations \eqref{eq:relative-disk-in-sector} and
\eqref{eq:epsilon-R-condition} show that the whole closed Rouch\'e disk lies in
the outer sector.
Choose also $\epsilon_\delta>0$ so that
\begin{equation}\label{eq:epsilon-delta-condition}
 (1+q)c_+\epsilon_\delta<\delta_1.
\end{equation}
Then the closed Rouch\'e disk is contained in
$\mathcal S(\delta_1,\vartheta_0)$ whenever
$|v|<\min\{\epsilon_R,\epsilon_\delta\}$.

On its boundary,
\[
 g_0(u)-v=(u-u_0)\int_0^1g_0'(u_0+t(u-u_0))\dd t,
\]
and therefore
\begin{equation}\label{eq:Rouche-g0}
 |g_0(u)-v|\ge\frac{3A}{4}r_v.
\end{equation}
By \eqref{eq:r-on-relative-disk} and \eqref{eq:M-final-choice},
$|r(u)|<|g_0(u)-v|$.  The function $g_0$ is univalent on the chosen inverse
neighborhood, so $g_0(u)-v$ has exactly one zero, $u_0=h(v)$, in the disk.
Rouch\'e's theorem therefore gives exactly one zero of $g(u)-v$ there; denote
it by $U(v)$.  Conversely, \eqref{eq:all-roots-near-h} and $M>C_2$ show that
every solution in $\mathcal S(\delta_1,\vartheta_0)$ lies strictly inside the
same disk.  Hence $U(v)$ is the unique local solution claimed in the lemma,
not merely the unique zero in a preselected moving disk.  Moreover,
\begin{equation}\label{eq:U-h-rough}
 U(v)-h(v)=O\bigl(v^4(1+|\Log v|)\bigr).
\end{equation}

\paragraph{Holomorphy.}
Rouch\'e's theorem counts zeros with multiplicity.  Since $g_0(u)-v$ has
one simple zero in the disk, the unique zero $U(v)$ of $g(u)-v$ is simple
as well; thus $g'(U(v))\ne0$.  At each fixed $v_*$, the inverse function
theorem gives a local holomorphic solution remaining in
$\mathcal S(\delta_1,\vartheta_0)$.  Uniqueness in that sector identifies
it with $U$, proving holomorphy throughout the $v$-sector.

\paragraph{The logarithmic term.}
Put $\Delta(v)=U(v)-h(v)$.  From
\eqref{eq:U-h-rough},
$\Delta(v)=O(v^4(1+|\Log v|))$.  Choose
$\epsilon_{\mathrm{exp}}>0$ so that, for
$0<|v|<\epsilon_{\mathrm{exp}}$, the quotients $h(v)/(v/A)$ and
$U(v)/(v/A)$ lie in a disk about $1$ whose arguments have absolute value
less than $\vartheta_1/2$.  This guarantees additivity of the principal
logarithms.  The segment from $h(v)$ to $U(v)$ lies in the Rouch\'e disk,
where $|u|\asymp|v|$ and
$1+|\Log u|\asymp1+|\Log v|$.  Since
$(u^4\Log u)'=u^3(4\Log u+1)$, integration of this derivative along
the segment and Taylor's integral remainder for $g_0$ give
\begin{align*}
 U(v)^4\Log U(v)-h(v)^4\Log h(v)
 &=O\bigl(|v|^7(1+|\Log v|)^2\bigr),\\
 g_0(U(v))-g_0(h(v))-g_0'(h(v))\Delta(v)
 &=O\bigl(|v|^8(1+|\Log v|)^2\bigr).
\end{align*}
Both quantities are $O(|v|^5(1+|\Log v|))$; the remainder in
Lemma~\ref{lem:g-expansion}, evaluated at $U(v)$, has the same order.
Using $g(U(v))=g_0(h(v))$ therefore yields
\begin{equation}\label{eq:Delta-equation}
 g_0'(h(v))\Delta(v)+D_4h(v)^4\Log h(v)
 =O\bigl(v^5(1+|\Log v|)\bigr).
\end{equation}
From \eqref{eq:h-expansion} and the preceding branch choice,
\[
 g_0'(h(v))=A+O(v),\qquad
 \Log h(v)=\Log v-\log A+O(v),
\]
so
\[
 h(v)^4\Log h(v)=A^{-4}v^4\Log v+c_4v^4
 +O\bigl(v^5(1+|\Log v|)\bigr).
\]
Equation \eqref{eq:Delta-equation} therefore yields
\[
 \Delta(v)=-\frac{D_4}{A^5}v^4\Log v+d_4v^4
 +O\bigl(v^5(1+|\Log v|)\bigr).
\]
Absorbing $d_4v^4$ into the fourth Taylor polynomial of $h$ proves
\eqref{eq:U-expansion}.

\paragraph{One final value of $\epsilon$.}
No further implicit shrinking is needed.  Take
\begin{equation}\label{eq:sectorial-final-epsilon}
 \epsilon:=\min\left\{
  \epsilon_h,\epsilon_{\mathrm{root}},\epsilon_R,\epsilon_\delta,
  \epsilon_{\mathrm{exp}},\frac{R_h}{2},\frac12
 \right\}.
\end{equation}
Each threshold depends only on previously fixed data and
is independent of $v$.  In particular, $\delta_1,C_1,C_2$ are fixed before
$M$; $\epsilon_R$ may depend on $M$, and $\epsilon_\delta$ on $\delta_1$.  Hence every estimate above holds simultaneously and uniformly for all
$v\in\mathcal S(\epsilon,\vartheta_1)$.  On any fixed smaller closed sector,
the same bounds are uniform up to its angular boundary, which proves the final
uniformity assertion.
\end{proof}

\subsection{Construction of the standard dented domain}

For $\epsilon>0$ and $0<\phi<\pi/2$, define
\begin{equation}\label{eq:Delta-domain-def}
 \Delta(\sigma,\epsilon,\phi)
 =\{y\in\C:|y|<\sigma+\epsilon,\ y\ne\sigma,
            |\Arg(y-\sigma)|>\phi\}.
\end{equation}
Because $y-\sigma=-\sigma v$, the local condition is equivalent to
$|\Arg v|<\pi-\phi$.

Fix $0<\vartheta_0<\vartheta_1<\phi<\pi/2$ and denote the radius
supplied by Lemma~\ref{lem:sectorial-inverse} by $\epsilon_{\rm inv}$.
Choose $0<\epsilon_y<\min\{\sigma\epsilon_{\rm inv},\sigma\}$.
Near $y=\sigma$, define

\begin{equation}\label{eq:Xloc-Sloc}
 X_{\mathrm{loc}}(y)=\rho\left(1-U\!\left(1-\frac y\sigma\right)\right),
 \qquad S_{\mathrm{loc}}(y)=F(X_{\mathrm{loc}}(y)).
\end{equation}
Then $\Phi(X_{\mathrm{loc}}(y))=y$ on a local dented sector $L$.

For real $y<\sigma$ sufficiently close to $\sigma$, put
$v=1-y/\sigma>0$.  The Stieltjes formula gives
$g(\overline u)=\overline{g(u)}$.  Since the local sector is invariant under
conjugation, the uniqueness in Lemma~\ref{lem:sectorial-inverse} implies
$U(v)=\overline{U(v)}$.  The expansion
\[
 U(v)=\frac vA+O(v^2)+O(v^4|\log v|)=\frac vA+o(v)
\]
and $A>0$ show explicitly that $U(v)>0$.  Hence
$X_{\mathrm{loc}}(y)=\rho(1-U(v))<\rho$ and tends to $\rho$.

We now verify, rather than assume, that the inverse $X$ coming from the origin
has the same real limit.  For $0<y<\sigma$, the combinatorial coefficients of
$S$ are nonnegative, so $S(y)>0$ and
\[
 X(y)=\frac{y}{S(y)^2}>0.
\]
For $0<x<\rho$, the Taylor coefficients of $F$ are nonnegative and
$f_0=f_1=1$, whence $F(x)>0$ and $F'(x)>0$.  Therefore
\begin{equation}\label{eq:Phi-real-positive}
 \Phi'(x)=F(x)^2+2xF(x)F'(x)>0\qquad(0<x<\rho).
\end{equation}
Moreover, $\Phi(0)=0$ and the endpoint expansion gives
$\Phi(x)\to\rho F_0^2=\sigma$ as $x\uparrow\rho$.  Thus $\Phi$ maps
$[0,\rho)$ strictly increasingly onto $[0,\sigma)$, and its unique positive
inverse satisfies
\begin{equation}\label{eq:X-real-limit}
 X(y)\uparrow\rho\qquad(y\uparrow\sigma).
\end{equation}
For real $y<\sigma$ sufficiently close to $\sigma$, the preceding estimates
also give $0<X_{\mathrm{loc}}(y)<\rho$ and
$\Phi(X_{\mathrm{loc}}(y))=y$.  Since \eqref{eq:Phi-real-positive} shows that
$\Phi$ maps $[0,\rho)$ strictly increasingly onto $[0,\sigma)$,
$X_{\mathrm{loc}}(y)$ is the unique positive preimage of $y$.  Hence
$X_{\mathrm{loc}}(y)=X(y)$ and $S_{\mathrm{loc}}(y)=S(y)$ there.  The identity
theorem extends equality to the relevant connected overlap with $\B_\sigma$.

\begin{theorem}[$\Delta$-analyticity of $S$]\label{thm:S-delta}
There exist $\epsilon>0$ and $0<\phi<\pi/2$ such that $S$ has a unique
single-valued holomorphic continuation to
$\Delta(\sigma,\epsilon,\phi)$.
\end{theorem}

\begin{proof}
Use the already fixed $\epsilon_y$ and $\phi$, and set
\begin{equation}\label{eq:L-local}
 L=\{y:0<|y-\sigma|<\epsilon_y,
          |\Arg(y-\sigma)|>\phi\}.
\end{equation}
The function $S_{\mathrm{loc}}$ is holomorphic on $L$.  Put
\[
 K=\{y:|y|=\sigma,\ |y-\sigma|\ge\epsilon_y/2\}.
\]
By Theorem~\ref{thm:nonprincipal-continuation} together with the compatible
finite-disk construction in Appendix~\ref{app:monodromy}, specifically
Lemma~\ref{lem:disk-disk} and \eqref{eq:background-cover}, an open neighborhood
$N\supset K$ carries a single holomorphic continuation of $S$.  Since $K$ is compact,
\[
 d:=\dist(K,\C\setminus N)>0.
\]
Choose
\begin{equation}\label{eq:epsilon-choice}
 0<\epsilon<\min\{d/2,\epsilon_y/4\}.
\end{equation}
Then
\begin{equation}\label{eq:Delta-cover}
 \Delta(\sigma,\epsilon,\phi)\subset\B_\sigma\cup N\cup L.
\end{equation}
Indeed, consider a point in the $\Delta$-domain with
$\sigma\le|y|<\sigma+\epsilon$.  If $|y-\sigma|<\epsilon_y$, it lies in $L$.
Otherwise let $y^*=\sigma y/|y|$.  Then
$|y-y^*|=|y|-\sigma<\epsilon$, while
\[
 |y^*-\sigma|\ge|y-\sigma|-|y-y^*|
 \ge\epsilon_y-\epsilon>\epsilon_y/2.
\]
Thus $y^*\in K$ and $\dist(y,K)<d$, so $y\in N$.

The local functions on $\B_\sigma$, $N$, and $L$ represent the same analytic
function element at zero.  The only remaining issue is single-valued gluing.  In
Appendix~\ref{app:monodromy} we prove explicitly that the domain in
\eqref{eq:Delta-domain-def} is star-shaped with respect to zero and hence
simply connected; we construct, along every path from zero, a finite chain of
open disks on which adjacent function elements agree.  The monodromy theorem
then gives a unique single-valued continuation on the whole domain.
\end{proof}

\section{Exact singular expansion of $S$ at $y=\sigma$}
\label{sec:S-singular}

Put $v=1-y/\sigma$.  From \eqref{eq:g0-D4} and
\eqref{eq:U-expansion}, the coefficient of $v^4\Log v$ in $U(v)$ is
\begin{equation}\label{eq:U-log-coeff}
 -\frac{D_4}{A^5}=\frac{2B}{F_0A^5}.
\end{equation}
Thus for a polynomial $H_4$ of degree at most four,
\begin{equation}\label{eq:U-final}
 U(v)=H_4(v)+Ev^4\Log v
 +O\bigl(v^5(1+|\Log v|)\bigr),
 \qquad E=\frac{2B}{F_0A^5},
\end{equation}
with $H_4(0)=0$, $H_4'(0)=A^{-1}$.

By \eqref{eq:F-endpoint},
\begin{equation}\label{eq:S-compose-U}
 S(y)=P(U(v))+BU(v)^4\Log U(v)
 +O\bigl(U(v)^5(1+|\Log U(v)|)\bigr).
\end{equation}
Since $U(v)=v/A+O(v^2)+O(v^4\Log v)$,
\begin{equation}\label{eq:Log-U}
 U(v)=O(v),\qquad \Log U(v)=\Log v-\log A+O(v),
\end{equation}
and the final remainder in \eqref{eq:S-compose-U} is
$O(v^5(1+|\Log v|))$.

The analytic part contributes
\begin{equation}\label{eq:P-U-expand}
 P(U(v))=P(H_4(v))+P'(0)Ev^4\Log v
 +O\bigl(v^5(1+|\Log v|)\bigr),
\end{equation}
whereas
\begin{equation}\label{eq:U4LogU}
 U(v)^4\Log U(v)=A^{-4}v^4\Log v+c_*v^4
 +O\bigl(v^5(1+|\Log v|)\bigr).
\end{equation}
The total logarithmic coefficient is therefore
\begin{align}
 P'(0)E+BA^{-4}
 &=-\rho F'(\rho)\frac{2B}{F_0A^5}+\frac{B}{A^4}\notag\\
 &=\frac{B}{A^5}\left(A-\frac{2\rho F'(\rho)}{F_0}\right)
 =\frac{B}{A^5}=-\frac1{\pi A^5}.
 \label{eq:total-log-coeff}
\end{align}
This calculation is essential: the logarithmic term in the inverse $U$ passes
through the analytic part $P$ and contributes at the same order.  The final
coefficient is not obtained by the linear replacement $u\sim v/A$ alone.

\begin{theorem}[Full leading singular expansion]\label{thm:S-singular}
There exists a polynomial $Q_4$ of degree at most four such that, uniformly as
$y\to\sigma$ in every fixed smaller $\Delta$-sector,
\begin{align}
 S(y)={}&Q_4\!\left(1-\frac y\sigma\right)
 -\frac1{\pi A^5}\left(1-\frac y\sigma\right)^4
   \Log\!\left(1-\frac y\sigma\right)\notag\\
 &+O\!\left(\left(1-\frac y\sigma\right)^5
 \left[1+\left|\Log\!\left(1-\frac y\sigma\right)\right|\right]\right).
 \label{eq:S-singular-full}
\end{align}
Moreover,
\begin{equation}\label{eq:Q4-data}
 Q_4(0)=F_0,
 \qquad Q_4'(0)=-\frac{\rho F'(\rho)}{A}.
\end{equation}
In particular, $y=\sigma$ is a genuine logarithmic branch point.
\end{theorem}

\begin{proof}
Equations \eqref{eq:S-compose-U}--\eqref{eq:total-log-coeff} identify all
nonanalytic terms through degree four.  Absorb the fourth Taylor polynomial of
$P(H_4(v))$, the analytic term $Bc_*v^4$, and all other analytic terms of
degree at most four into $Q_4(v)$.  Every remaining term has order
$O(v^5(1+|\Log v|))$.  The constant term is $P(0)=F_0$.  Since
$U(v)=v/A+O(v^2)$, the linear term is $P'(0)v/A$, which gives
\eqref{eq:Q4-data}.
\end{proof}

For $n\ge5$, direct expansion gives the exact identity
\begin{equation}\label{eq:transfer-log-exact}
 [z^n](1-z)^4\Log(1-z)
 =-\frac{24}{n(n-1)(n-2)(n-3)(n-4)}\sim-24n^{-5}.
\end{equation}
The $\Delta$-domain transfer theorem
\cite{FlajoletOdlyzko,FlajoletSedgewick} applied to
\eqref{eq:S-singular-full} therefore reproduces the coefficient asymptotic as
follows.  If $R_S(y)$ denotes the remainder in that expansion, then for any
fixed $0<\varepsilon<1$,
\[
 |v|^5(1+|\Log v|)=O(|v|^{5-\varepsilon})
 \qquad(v\to0)
\]
uniformly in every fixed smaller $\Delta$-sector.  The big-$O$ form of the
$\Delta$-transfer theorem thus gives
\[
 [y^n]R_S(y)=O(\sigma^{-n}n^{-6+\varepsilon})
 =o(\sigma^{-n}n^{-5}).
\]
Together with \eqref{eq:transfer-log-exact} (and the fact that $Q_4$ has no
coefficients beyond degree four), this proves
\eqref{eq:S-coeff-main}.  This is an a posteriori consistency check.  The
coefficient estimate used to prove boundary regularity was the independent
Lagrange--Hankel proof, so there is no circularity.

\section{From $S$ to the canonical generating function $S_3^{[4]}$}
\label{sec:canonical}

Recall that
\begin{equation}\label{eq:canonical-recall}
 S_3^{[4]}(z)=(1-z)G(\vartheta(z)),
 \qquad G(y)=S(y)-1-y,
 \qquad
 \vartheta(z)=\frac{z^6}{(1-z)^2(1-z^2+z^6)}.
\end{equation}
Let $\eta$ be the unique solution in $(0,1/2)$ of $\vartheta(\eta)=\sigma$.

\subsection{The unique dominant preimage}

\begin{lemma}\label{lem:eta-unique}
The equation $\vartheta(x)=\sigma$ has a unique solution in $(0,1/2)$,
namely
\begin{equation}\label{eq:eta-values}
 \eta=0.49340718057613087519\ldots,
 \qquad
 \vartheta'(\eta)=1.25142977847290709269\ldots>0.
\end{equation}
Moreover, for $|z|\le\eta$,
\begin{equation}\label{eq:theta-strict}
 |\vartheta(z)|\le\sigma,
\end{equation}
with equality if and only if $z=\eta$.
\end{lemma}

\begin{proof}
For $0<x\le1/2$ the denominator of $\vartheta$ is positive and
\begin{equation}\label{eq:theta-log-derivative}
 \frac{\vartheta'(x)}{\vartheta(x)}
 =\frac6x+\frac2{1-x}+\frac{2x-6x^5}{1-x^2+x^6}>0,
\end{equation}
because $1-3x^4>0$.  Thus $\vartheta$ is strictly increasing.  Since
$\vartheta(1/2)=4/49$, it remains to compare this value with $\sigma$ without
using decimals.  Since $165\pi-512>0$, \eqref{eq:constants-exact} gives
\begin{align*}
 \sigma<\frac4{49}
 &\Longleftrightarrow
 49(165\pi-512)^2<225\pi^2\\
 &\Longleftrightarrow
 7(165\pi-512)<15\pi
 \Longleftrightarrow
 \pi<\frac{896}{285}.
\end{align*}
The last inequality follows from
$\pi<22/7<896/285$.  Hence there is exactly one solution in $(0,1/2)$.
Directed rational interval evaluation gives the certified enclosures
\[
 0.49340718057613087519<\eta
 <0.49340718057613087520,
\]
and
\[
 1.25142977847290709269<\vartheta'(\eta)
 <1.25142977847290709270,
\]
which justify the rounded values in \eqref{eq:eta-values}.
For completeness, use Machin's identity
$\pi=16\arctan(1/5)-4\arctan(1/239)$ and bound each arctangent between
its $90$-term alternating partial sum and that sum plus the next term; for
$\arctan(1/d)$ the error is at most $1/(181d^{181})$.  Substitution gives
rational intervals for $F_0,A,\sigma$.  In particular, the initial bracket is
certified without floating-point comparisons by the exact rational chain
\[
 \vartheta(49/100)<\frac{69}{1000}
 <\frac{72924}{10^6}<\sigma<\frac4{49}=\vartheta(1/2),
\]
where the lower bound for $\sigma$ comes from the same Machin interval and the
upper bound was proved above.  Starting from $[49/100,1/2]$, perform $200$
rational bisections of $\vartheta(x)=\sigma$.  At each performed step the exact
rational value of $\vartheta$ at the midpoint lies strictly outside the
rational interval enclosing $\sigma$, so every branch decision is exact.  This
gives the displayed interval for $\eta$.  Substituting it into
\eqref{eq:theta-log-derivative}, with outward rounding only at the final
step, gives the stated interval for $\vartheta'(\eta)$.  The same interval
calculation applied to \eqref{eq:canonical-final} below certifies the
displayed bound for $C'$.  The accompanying ancillary file
\texttt{certify\_constants.py} is a dependency-free exact-rational
implementation of these interval calculations.

For $|z|\le\eta<1/2$, the denominator is nonzero because
$|z^2-z^6|<1$.  By the maximum-modulus principle it suffices to take
$z=\eta e^{it}$.  Clearly $|1-z|\ge1-\eta$, with equality only at $t=0$.
Put $x=\cos2t$.  A direct calculation gives
\begin{align}
 |1-z^2+z^6|^2-(1-\eta^2+\eta^6)^2
 =2\eta^2(1-x)\bigl[1+2\eta^6(1+x)
 -\eta^4(4x^2+4x+1)\bigr].\label{eq:denominator-difference}
\end{align}
For $-1\le x\le1$ and $\eta<1/2$, the bracket is bounded below by
$1-9\eta^4>1-9/16>0$.  Hence
$|1-z^2+z^6|\ge1-\eta^2+\eta^6$.  Combining both denominator estimates gives
$|\vartheta(z)|\le\vartheta(\eta)=\sigma$ on the boundary, and equality can
occur there only at $z=\eta$.  Since $\vartheta$ is nonconstant, the strong
maximum-modulus principle excludes equality at every interior point.  Hence
equality on the closed disk occurs only at $z=\eta$.
\end{proof}

\subsection{$\Delta$-analyticity and the local expansion of $S_3^{[4]}$}

Set
\begin{equation}\label{eq:Lambda}
 \Lambda:=\frac{\eta\vartheta'(\eta)}{\sigma}>0,
 \qquad t:=1-\frac z\eta.
\end{equation}
Taylor expansion gives
\begin{equation}\label{eq:theta-local}
 1-\frac{\vartheta(z)}\sigma=\Lambda t+O(t^2).
\end{equation}

We give the complete compactness argument that produces a standard
$\Delta$-domain.  Choose parameters $\epsilon_S>0$ and
$0<\phi<\pi/2$ for the continuation of $S$ furnished by
Theorem~\ref{thm:S-delta}, and then fix
\begin{equation}\label{eq:canonical-angle-choice}
 \phi<\psi<\frac\pi2.
\end{equation}
Put
\[
 w(z):=1-\frac{\vartheta(z)}\sigma.
\]
Equation \eqref{eq:theta-local} says
$w(z)=\Lambda t(1+O(t))$, with $\Lambda>0$.  Hence there is $\tau_0>0$ such
that, whenever
\begin{equation}\label{eq:canonical-local-z-sector}
 0<|t|<\tau_0,\qquad |\Arg t|<\pi-\psi,
\end{equation}
we have
\[
 |\Arg w(z)-\Arg t|<\psi-\phi,
 \qquad |\vartheta(z)-\sigma|<\epsilon_S/2.
\]
After decreasing $\tau_0$ if necessary, $\vartheta$ has no pole there and
$|\vartheta(z)|<\sigma+\epsilon_S$.  Since
$z-\eta=-\eta t$, condition
$|\Arg(z-\eta)|>\psi$ is equivalent to the angular condition in
\eqref{eq:canonical-local-z-sector}.  It follows that the local dented sector
\begin{equation}\label{eq:canonical-local-map}
 0<|z-\eta|<r_0:=\eta\tau_0,\qquad
 |\Arg(z-\eta)|>\psi
\end{equation}
is mapped by $\vartheta$ into
$\Delta(\sigma,\epsilon_S,\phi)$.

It remains to control points away from $\eta$.  The compact set
\begin{equation}\label{eq:canonical-compact-K}
 K_\eta:=\{z:|z|\le\eta,\ |z-\eta|\ge r_0/2\}
\end{equation}
does not contain $\eta$.  By \eqref{eq:theta-strict} and compactness, there is
$\delta_\eta>0$ such that
\begin{equation}\label{eq:canonical-compact-margin}
 |\vartheta(z)|\le\sigma-2\delta_\eta
 \qquad(z\in K_\eta).
\end{equation}
Continuity gives an open neighborhood $N_\eta\supset K_\eta$ on which
$|\vartheta(z)|<\sigma-\delta_\eta$.  Put
\[
 d_\eta:=\dist(K_\eta,\C\setminus N_\eta)>0.
\]
Also choose $e_{\mathrm{pole}}>0$ so that $\vartheta$ is holomorphic on
$|z|<\eta+e_{\mathrm{pole}}$.  Finally take
\begin{equation}\label{eq:canonical-epsilon-choice}
 0<\epsilon_1<\min\left\{
 d_\eta,\frac{r_0}{4},e_{\mathrm{pole}}
 \right\}.
\end{equation}

We claim that every point of $\Delta(\eta,\epsilon_1,\psi)$ either belongs to
the local region \eqref{eq:canonical-local-map} or to $N_\eta$.  Indeed, let
$z$ be outside the local region.  Then $|z-\eta|\ge r_0$.  If $|z|\le\eta$,
then $z\in K_\eta\subset N_\eta$.  If
$\eta<|z|<\eta+\epsilon_1$, put
$z^*=\eta z/|z|$.  We have
\[
 |z-z^*|=|z|-\eta<\epsilon_1,
 \qquad
 |z^*-\eta|\ge |z-\eta|-|z-z^*|
 >r_0-\epsilon_1>\frac{r_0}{2}.
\]
Thus $z^*\in K_\eta$ and $\dist(z,K_\eta)<\epsilon_1<d_\eta$, whence
$z\in N_\eta$.  On $N_\eta$ the image lies in $\B_\sigma$, while the local
region maps into the already constructed $\Delta$-domain of $G$.  On the overlap, both expressions evaluate the same already single-valued function $G$ on its $\Delta$-domain; hence they agree.  Therefore $G(\vartheta(z))$, and hence $S_3^{[4]}(z)$, is single-valued
and holomorphic on the standard domain
$\Delta(\eta,\epsilon_1,\psi)$.

Using
\begin{equation}\label{eq:compose-log}
 (\Lambda t+O(t^2))^4\Log(\Lambda t+O(t^2))
 =\Lambda^4t^4\Log t+c_\Lambda t^4
 +O\bigl(t^5(1+|\Log t|)\bigr),
\end{equation}
we obtain the strengthened bridge to the canonical series.

\begin{theorem}[Canonical local expansion]\label{thm:canonical-expansion}
There exist $\epsilon_1>0$, $0<\psi<\pi/2$, and a polynomial
$\widetilde Q_4$ of degree at most four such that $S_3^{[4]}$ is holomorphic on
$\Delta(\eta,\epsilon_1,\psi)$ and
\begin{align}
 S_3^{[4]}(z)={}&\widetilde Q_4\!\left(1-\frac z\eta\right)
 -\kappa\left(1-\frac z\eta\right)^4
       \Log\!\left(1-\frac z\eta\right)\notag\\
 &+O\!\left(\left(1-\frac z\eta\right)^5
 \left[1+\left|\Log\!\left(1-\frac z\eta\right)\right|\right]\right),
 \label{eq:canonical-expansion}
\end{align}
where
\begin{equation}\label{eq:kappa}
 \kappa=\frac{1-\eta}{\pi A^5}
 \left(\frac{\eta\vartheta'(\eta)}{\sigma}\right)^4>0.
\end{equation}
\end{theorem}

\begin{proof}
The preceding mapping argument proves $\Delta$-analyticity.  In
$(1-z)G(\vartheta(z))$, the prefactor $1-z=(1-\eta)+O(t)$ is analytic and
nonzero at $\eta$.  Subtracting $1+y$ from $S(y)$ does not alter its
nonanalytic term.  Thus the only order-four nonanalytic term is
\[
 -\frac{1-\eta}{\pi A^5}\Lambda^4t^4\Log t.
\]
All analytic terms through order four are absorbed into $\widetilde Q_4$; the
remainder follows from \eqref{eq:compose-log}.
\end{proof}

\begin{theorem}[Canonical coefficient asymptotics]\label{thm:canonical-coeff}
Let $S_3^{[4]}(n)=[z^n]S_3^{[4]}(z)$.  Then
\begin{equation}\label{eq:canonical-final}
 S_3^{[4]}(n)\sim C'n^{-5}\eta^{-n},
 \qquad
 C'=\frac{24(1-\eta)}{\pi A^5}
 \left(\frac{\eta\vartheta'(\eta)}{\sigma}\right)^4
 =7892.16205625817\ldots.
\end{equation}
The same directed interval evaluation used in \eqref{eq:eta-values} gives
\[
 7892.1620562581748200<C'<7892.1620562581748201,
\]
so the displayed decimal is certified rather than inferred from an
uncontrolled floating-point computation.
Equivalently,
\begin{equation}\label{eq:Cprime-alternative}
 C'=(1-\eta)\vartheta'(\eta)^4
 \frac{820125\pi^7}{256(256-81\pi)^5(165\pi-512)^3}\,4!\,\eta^4.
\end{equation}
\end{theorem}

\begin{proof}
Let $R_C(z)$ be the remainder in \eqref{eq:canonical-expansion}.  For any fixed
$0<\varepsilon<1$, the estimate
$|t|^5(1+|\Log t|)=O(|t|^{5-\varepsilon})$ and the big-$O$
$\Delta$-transfer theorem give
\[
 [z^n]R_C(z)=O(\eta^{-n}n^{-6+\varepsilon})
 =o(\eta^{-n}n^{-5}).
\]
Apply \eqref{eq:transfer-log-exact} to the remaining nonanalytic term in
\eqref{eq:canonical-expansion}; the polynomial contributes nothing for
$n>4$.  The coefficient of $-t^4\Log t$ is $\kappa$, so
\[
 [z^n]S_3^{[4]}(z)\sim24\kappa\eta^{-n}n^{-5}.
\]
Substituting
\eqref{eq:constants-exact} and using
$1280-405\pi=5(256-81\pi)$ gives \eqref{eq:Cprime-alternative}.
\end{proof}

\section{Numerical validation of the coefficient asymptotics}
\label{sec:numerical-validation}

We compare the two asymptotic formulas with coefficients computed from the
exact generating-function identities, independently of the asymptotic
approximations.  Write $a_n=[z^n]S_3^{[4]}(z)$ and define
\begin{equation}\label{eq:numerical-ratios}
 R_S(n):=\frac{S(n)}{\dfrac{24}{\pi A^5}\sigma^{-n}n^{-5}},
 \qquad
 R_C(n):=\frac{a_n}{C'\eta^{-n}n^{-5}},
\end{equation}
where $A,\sigma,\eta$, and $C'$ are the constants already determined in
\eqref{eq:constants-exact} and \eqref{eq:canonical-final}.
Theorems~\ref{thm:S-coeff-main} and \ref{thm:canonical-coeff} assert that
both ratios tend to $1$.

\begin{figure}[!htbp]
\centering
\includegraphics[width=\textwidth]{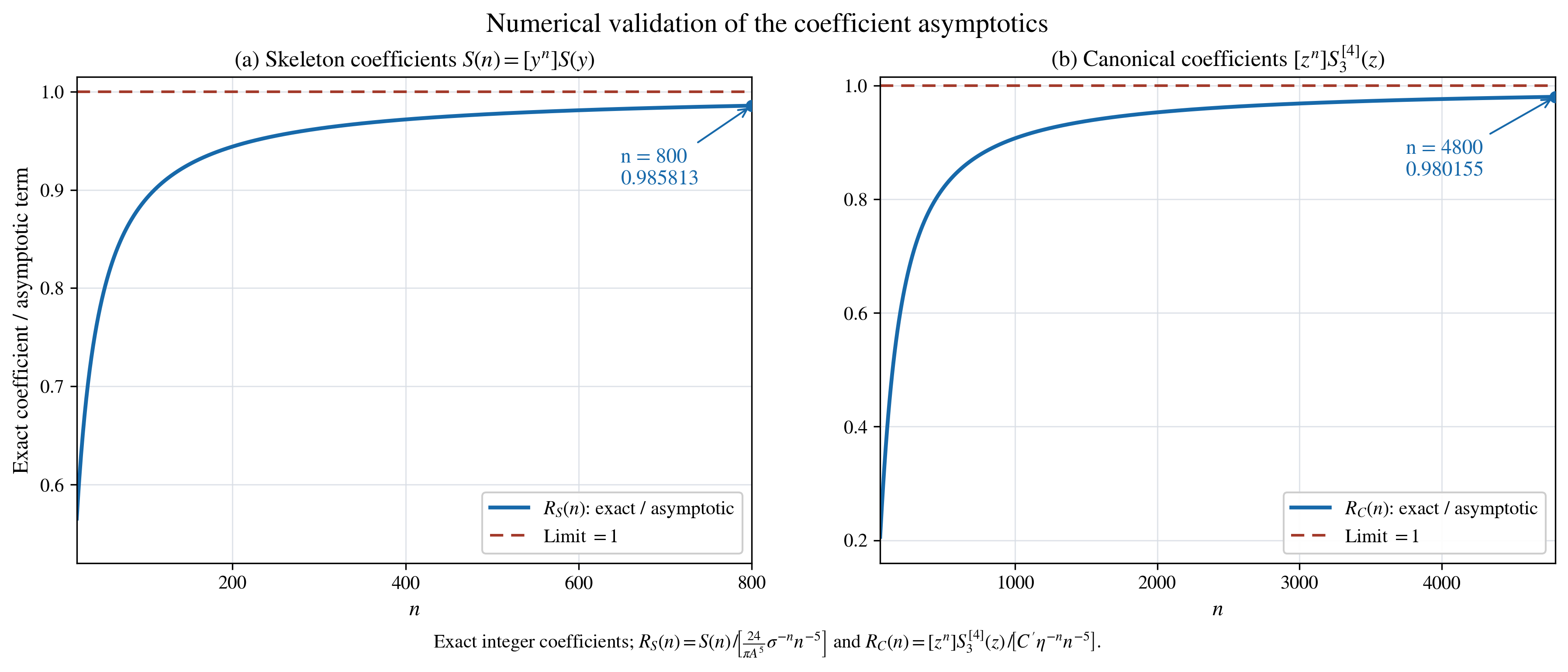}
\caption{Numerical validation of \eqref{eq:S-coeff-main} and
\eqref{eq:canonical-final}.  The left panel plots $R_S(n)$ for
$20\le n\le800$, and the right panel plots $R_C(n)$ for
$50\le n\le4800$.  Every integer index in these ranges is included.
The dashed horizontal line marks the predicted limit $1$.}
\label{fig:asymptotic-ratio-validation}
\end{figure}

The computation uses arbitrary-precision integer arithmetic for the
coefficients.  First, the matching coefficients are generated from the
Catalan determinant in Appendix~\ref{app:catalan}:
\[
 [z^j]F(z)=\operatorname{Cat}_j\operatorname{Cat}_{j+2}
 -\operatorname{Cat}_{j+1}^{\,2},
 \qquad \operatorname{Cat}_j=\frac{1}{j+1}\binom{2j}{j}.
\]
Equation~\eqref{eq:LB-exact} then gives $S(n)$ by exact truncated
power-series arithmetic, and the canonical identity
\eqref{eq:canonical-recall} gives
\[
 a_n=\sum_{k=2}^{\lfloor n/6\rfloor}
 S(k)[z^n]\bigl((1-z)\vartheta(z)^k\bigr),
 \qquad
 (1-z)\vartheta(z)^k
 =z^{6k}(1-z)^{1-2k}(1-z^2+z^6)^{-k}.
\]
Thus only $S(k)$ with $k\le800$ are needed to compute all canonical
coefficients through $n=4800$.  All series computations are exact;
80-digit decimal arithmetic is used only to evaluate the constants and
normalize the resulting integer coefficients.  No asymptotic constant or
exponent is fitted to the coefficient data.

\begin{table}[htbp]
\centering
\begin{tabular}{rcrc}
\hline
$n$ & $R_S(n)$ & $n$ & $R_C(n)$\\
\hline
100 & 0.8907298489 &  600 & 0.8493414556\\
200 & 0.9441428881 & 1200 & 0.9222807255\\
400 & 0.9717750294 & 2400 & 0.9605848530\\
800 & 0.9858126854 & 4800 & 0.9801548347\\
\hline
\end{tabular}
\caption{Representative ratios of exact coefficients to the predicted
asymptotic leading terms.  The two pairs of columns use different ranges
of $n$; the displayed ratios are rounded to ten decimal places.}
\label{tab:numerical-ratios}
\end{table}

In the entire plotted ranges, both ratios are strictly increasing and
remain below $1$.  At the largest computed indices, we obtain
\[
 R_S(800)=0.9858126854\ldots,
 \qquad R_C(4800)=0.9801548347\ldots.
\]
Their deviations from $1$ are approximately
$1.42\%$ and $1.98\%$, respectively.  Table~\ref{tab:numerical-ratios} and
Figure~\ref{fig:asymptotic-ratio-validation} show systematic improvement
as $n$ increases.  These exact-coefficient experiments support both the
predicted growth scales and the leading constants, with visible finite-$n$
corrections.  The limit statements themselves follow from the analytic
proofs; the experiments provide a separate numerical consistency check.

\section{Conclusion}

The proof has the noncircular chain
\[
 \begin{gathered}
 \text{Lagrange--Hankel asymptotics}
 \Longrightarrow S,S'\text{ regular on }|y|\le\sigma
 \Longrightarrow \Phi:\Omega\xrightarrow{\sim}\B_\sigma,\\
 \Phi:\Omega\xrightarrow{\sim}\B_\sigma
 \Longrightarrow \text{continuation through }|y|=\sigma,\ y\ne\sigma
 \Longrightarrow \text{sectorial Rouch\'e inversion},\\
 \text{sectorial inversion}
 \Longrightarrow S\in\mathcal O(\Delta(\sigma,\epsilon,\phi))
 \Longrightarrow \eqref{eq:S-singular-full}
 \Longrightarrow \eqref{eq:S-coeff-main},\\
 \eqref{eq:S-singular-full}+\eqref{eq:canonical-recall}
 \Longrightarrow S_3^{[4]}\in\mathcal O(\Delta(\eta,\epsilon_1,\psi))
 \Longrightarrow \eqref{eq:canonical-final}.
 \end{gathered}
\]
The exact leading constants are
\[
 \frac{24}{\pi A^5}=3.030955878514041\ldots,
 \qquad C'=7892.16205625817\ldots,
\]
with associated positive singularities
\[
 \sigma=0.0729243577550381\ldots,
 \qquad
 \eta=0.493407180576130875\ldots.
\]

\appendix

\section{From the hypergeometric coefficients to the Catalan determinant}
\label{app:catalan}

From
\[
 {}_2F_1(a,b;c;z)=\sum_{k\ge0}\frac{(a)_k(b)_k}{(c)_k k!}z^k,
\]
the coefficient of $z^n$ in \eqref{eq:F-def} is
\begin{equation}\label{eq:fn-hypergeom}
 f_n=\frac14\left[
 -16^{n+2}\frac{(-1/2)_{n+2}^2}{((n+2)!)^2}
 -2\cdot16^{n+1}\frac{(-1/2)_{n+1}(3/2)_{n+1}}
 {(3)_{n+1}(n+1)!}
 \right].
\end{equation}
Using
\[
 \left(-\frac12\right)_k=-\frac{(2k-2)!}{2^{2k-1}(k-1)!},
 \qquad
 \left(\frac32\right)_k=\frac{(2k+1)!}{4^k k!},
 \qquad (3)_k=\frac{(k+2)!}{2},
\]
one obtains
\begin{equation}\label{eq:fn-factorial}
 f_n=\frac{12(2n)!(2n+1)!}{(n!)^2(n+2)!(n+3)!}.
\end{equation}
On the other hand,
\[
 \frac{C_{n+1}}{C_n}=\frac{2(2n+1)}{n+2},
\]
and hence
\begin{align*}
 C_nC_{n+2}-C_{n+1}^2
 &=C_nC_{n+1}\left(\frac{2(2n+3)}{n+3}
                    -\frac{2(2n+1)}{n+2}\right)\\
 &=\frac{6C_nC_{n+1}}{(n+2)(n+3)}
 =\frac{12(2n)!(2n+1)!}{(n!)^2(n+2)!(n+3)!}.
\end{align*}
This proves \eqref{eq:catalan-det}.

\section{Convergent hypergeometric expansions at the finite endpoint and the
logarithmic coefficient}
\label{app:logcoeff}

This appendix derives the two convergent power series underlying
\eqref{eq:H1-local}--\eqref{eq:H2-local} and then performs the cancellation
that produces the coefficient $B=-1/\pi$ in \eqref{eq:F-endpoint}.
Throughout,
\[
 |u|<1,\qquad |\Arg u|<\pi,
\]
and $\Log u$ denotes the corresponding principal branch.

\subsection{The complete integer-difference connection formula}

Let
\[
 c=a+b+\ell,\qquad \ell\in\N_0.
\]
DLMF \cite[\S15.8(ii), Eq.~(15.8.10)]{DLMF} gives the connection formula for
Olver's regularized hypergeometric function.  Multiplying that formula by
$\Gamma(c)$ and putting $z=1-u$ gives the following exact expansion for the
ordinary Gauss function.  We use parameters for which the displayed gamma and digamma values are finite; both triples below satisfy this condition.
\begin{align}
 {}_2F_1(a,b;c;1-u)
 ={}&
 \frac{\Gamma(c)}{\Gamma(a+\ell)\Gamma(b+\ell)}
 \sum_{k=0}^{\ell-1}
 \frac{(a)_k(b)_k(\ell-k-1)!}{k!}(-u)^k
 \notag\\
 &-\frac{\Gamma(c)(-u)^\ell}{\Gamma(a)\Gamma(b)}
 \sum_{k=0}^{\infty}
 \alpha_{\ell,k}(a,b)u^k
 \left(\Log u+\beta_{\ell,k}(a,b)\right),
 \label{eq:endpoint-connection}
\end{align}
where the first sum is empty when $\ell=0$, and
\begin{align}
 \alpha_{\ell,k}(a,b)
 &:=
 \frac{(a+\ell)_k(b+\ell)_k}{k!(k+\ell)!},
 \label{eq:endpoint-alpha}\\
 \beta_{\ell,k}(a,b)
 &:=
 -\psi(k+1)-\psi(k+\ell+1)
 +\psi(a+k+\ell)+\psi(b+k+\ell).
 \label{eq:endpoint-beta}
\end{align}
Equation \eqref{eq:endpoint-connection} is a convergent connection formula,
not merely a formal asymptotic expansion.

To display its two analytic parts, define
\begin{align}
 \mathcal L_{a,b,\ell}(u)
 &:=
 -\frac{\Gamma(c)(-u)^\ell}{\Gamma(a)\Gamma(b)}
 \sum_{k=0}^{\infty}\alpha_{\ell,k}(a,b)u^k,
 \label{eq:endpoint-log-series}\\
 \mathcal A_{a,b,\ell}(u)
 &:=
 \frac{\Gamma(c)}{\Gamma(a+\ell)\Gamma(b+\ell)}
 \sum_{k=0}^{\ell-1}
 \frac{(a)_k(b)_k(\ell-k-1)!}{k!}(-u)^k
 \notag\\
 &\quad
 -\frac{\Gamma(c)(-u)^\ell}{\Gamma(a)\Gamma(b)}
 \sum_{k=0}^{\infty}
 \alpha_{\ell,k}(a,b)\beta_{\ell,k}(a,b)u^k.
 \label{eq:endpoint-regular-series}
\end{align}
Then \eqref{eq:endpoint-connection} becomes
\begin{equation}
 {}_2F_1(a,b;c;1-u)
 =\mathcal A_{a,b,\ell}(u)
 {}+\mathcal L_{a,b,\ell}(u)\Log u.
 \label{eq:endpoint-two-series}
\end{equation}
Thus the coefficient of $\Log u$ and the nonlogarithmic part are represented
by two separate power series.

\subsection{Direct convergence verification}

For the parameter triples used below, the gamma-ratio estimate
\cite[\S5.11(iii)]{DLMF} gives
\begin{align}
 \alpha_{\ell,k}(a,b)
 &=
 \frac{\Gamma(k+a+\ell)\Gamma(k+b+\ell)}
 {\Gamma(a+\ell)\Gamma(b+\ell)
  \Gamma(k+1)\Gamma(k+\ell+1)}
 \notag\\
 &=
 \frac{k^{a+b+\ell-2}}
 {\Gamma(a+\ell)\Gamma(b+\ell)}
 \left(1+O(k^{-1})\right)
 =
 \frac{k^{c-2}}
 {\Gamma(a+\ell)\Gamma(b+\ell)}
 \left(1+O(k^{-1})\right).
 \label{eq:endpoint-alpha-asymptotic}
\end{align}
The standard fixed-shift expansion \cite[Eq.~(5.11.2)]{DLMF}
\[
 \psi(k+s)=\log k+\frac{s-\frac12}{k}+O(k^{-2})
 \qquad(k\to\infty)
\]
gives, after cancellation of the four $\log k$ terms,
\begin{align}
 \beta_{\ell,k}(a,b)
 &=\frac{
 -\frac12-(\ell+\frac12)
 +(a+\ell-\frac12)+(b+\ell-\frac12)}
 {k}
 +O(k^{-2})
 \notag\\
 &=\frac{a+b+\ell-2}{k}+O(k^{-2})
 =\frac{c-2}{k}+O(k^{-2}).
 \label{eq:endpoint-beta-asymptotic}
\end{align}

For the first hypergeometric function below, $c=1$, and hence
\[
 \alpha_{\ell,k}=O(k^{-1}),
 \qquad
 \alpha_{\ell,k}\beta_{\ell,k}=O(k^{-2}).
\]
For the second, $c=3$, and hence
\[
 \alpha_{\ell,k}=O(k),
 \qquad
 \alpha_{\ell,k}\beta_{\ell,k}=O(1).
\]
In the present two cases the leading constants in
\eqref{eq:endpoint-alpha-asymptotic}--\eqref{eq:endpoint-beta-asymptotic}
are nonzero.  Their nonvanishing proves that the radius is
exactly one; the upper bounds alone would only prove a radius at least one.
Hence the $k$th roots of the absolute values of the coefficients
in all four series tend to one, and the Cauchy--Hadamard formula gives radius
of convergence one.  Therefore both
\eqref{eq:endpoint-log-series} and
\eqref{eq:endpoint-regular-series} converge absolutely for $|u|<1$ and
locally uniformly there.  In particular,
\[
 \mathcal A_{a,b,\ell},\ \mathcal L_{a,b,\ell}
 \in\mathcal O(\{u:|u|<1\}).
\]

\subsection{The two hypergeometric functions}

To keep the notation consistent with
\eqref{eq:H1-local}--\eqref{eq:H2-local}, write
\[
 H_j(1-u)=P_j(u)+L_j(u)\Log u,\qquad j=1,2,
\]
where $P_j$ is the regular part and $L_j$ is the logarithmic coefficient.
Thus $P_j$ and $L_j$ are not two different hypergeometric functions: they
are the two convergent analytic factors of $H_j(1-u)$.

\paragraph{The first function.}
For
\[
 H_1(z)={}_2F_1\!\left(-\frac12,-\frac12;1;z\right),
\]
we have $a=b=-1/2$, $c=1$, and $\ell=2$.  Equations
\eqref{eq:endpoint-log-series}--\eqref{eq:endpoint-regular-series} give
\begin{align}
 L_1(u)
 &=-\frac{u^2}{4\pi}
 \sum_{k=0}^{\infty}
 \frac{(3/2)_k^2}{k!(k+2)!}u^k,
 \label{eq:L1-complete}\\
 P_1(u)
 &=\frac4\pi\left(1-\frac u4\right)
 -\frac{u^2}{4\pi}
 \sum_{k=0}^{\infty}
 \frac{(3/2)_k^2}{k!(k+2)!}
 \beta_{2,k}\!\left(-\frac12,-\frac12\right)u^k.
 \label{eq:P1-complete}
\end{align}
Both series converge for $|u|<1$.  The first three coefficients in
\eqref{eq:L1-complete} are
\[
 \frac{(3/2)_0^2}{0!\,2!}=\frac12,\qquad
 \frac{(3/2)_1^2}{1!\,3!}=\frac38,\qquad
 \frac{(3/2)_2^2}{2!\,4!}=\frac{75}{256}.
\]
Consequently,
\begin{equation}
 L_1(u)
 =-\frac{u^2}{8\pi}-\frac{3u^3}{32\pi}
 -\frac{75u^4}{1024\pi}+O(u^5).
 \label{eq:L1-first-terms}
\end{equation}

\paragraph{The second function.}
For
\[
 H_2(z)={}_2F_1\!\left(-\frac12,\frac32;3;z\right),
\]
we have $a=-1/2$, $b=3/2$, $c=3$, and $\ell=2$.  Hence
\begin{align}
 L_2(u)
 &=\frac{2u^2}{\pi}
 \sum_{k=0}^{\infty}
 \frac{(3/2)_k(7/2)_k}{k!(k+2)!}u^k,
 \label{eq:L2-complete}\\
 P_2(u)
 &=\frac{32}{15\pi}\left(1+\frac{3u}{4}\right)
 {}+\frac{2u^2}{\pi}
 \sum_{k=0}^{\infty}
 \frac{(3/2)_k(7/2)_k}{k!(k+2)!}
 \beta_{2,k}\!\left(-\frac12,\frac32\right)u^k.
 \label{eq:P2-complete}
\end{align}
Again, both series converge for $|u|<1$.  The first three coefficients in
\eqref{eq:L2-complete} are
\[
 \frac{(3/2)_0(7/2)_0}{0!\,2!}=\frac12,\qquad
 \frac{(3/2)_1(7/2)_1}{1!\,3!}=\frac78,\qquad
 \frac{(3/2)_2(7/2)_2}{2!\,4!}=\frac{315}{256}.
\]
It follows that
\begin{equation}
 L_2(u)
 =\frac{u^2}{\pi}+\frac{7u^3}{4\pi}
 +\frac{315u^4}{128\pi}+O(u^5).
 \label{eq:L2-first-terms}
\end{equation}
Equations \eqref{eq:L1-first-terms}--\eqref{eq:L2-first-terms} reproduce the
logarithmic parts in \eqref{eq:H1-local}--\eqref{eq:H2-local}, while
\eqref{eq:P1-complete}--\eqref{eq:P2-complete} justify the statement there
that $P_1$ and $P_2$ are holomorphic at zero.

\subsection{Cancellation in the combination defining $F$}

In \eqref{eq:F-def}, with $z=\rho(1-u)$ and $\rho=1/16$, the part involving
the two hypergeometric functions is
\[
 \frac{64}{(1-u)^2}
 \left[-H_1(1-u)-\frac{1-u}{8}H_2(1-u)\right].
\]
Its coefficient of $\Log u$ is therefore
\begin{equation}
 \frac{64}{(1-u)^2}
 \left[-L_1(u)-\frac{1-u}{8}L_2(u)\right].
 \label{eq:F-log-coefficient-combination}
\end{equation}
From \eqref{eq:L2-first-terms},
\begin{align*}
 \frac{1-u}{8}L_2(u)
 &=
 \frac18\left[
 \frac{u^2}{\pi}
 {}+\left(\frac74-1\right)\frac{u^3}{\pi}
 {}+\left(\frac{315}{128}-\frac74\right)\frac{u^4}{\pi}
 {}+O(u^5)\right]\\
 &=
 \frac{u^2}{8\pi}
 {}+\frac{3u^3}{32\pi}
 {}+\frac{91u^4}{1024\pi}
 {}+O(u^5).
\end{align*}
Combining this with \eqref{eq:L1-first-terms} gives the exact cancellations
at orders two and three:
\begin{align*}
 -L_1(u)-\frac{1-u}{8}L_2(u)
 &=
 \left(\frac{u^2}{8\pi}+\frac{3u^3}{32\pi}
 +\frac{75u^4}{1024\pi}\right)\\
 &\quad-
 \left(\frac{u^2}{8\pi}+\frac{3u^3}{32\pi}
 +\frac{91u^4}{1024\pi}\right)
 +O(u^5)\\
 &=-\frac{u^4}{64\pi}+O(u^5).
\end{align*}
Finally,
\[
 \frac{64}{(1-u)^2}
 =64\left(1+2u+O(u^2)\right),
\]
so \eqref{eq:F-log-coefficient-combination} equals
\[
 -\frac{u^4}{\pi}+O(u^5).
\]
Therefore the logarithmic coefficient in \eqref{eq:F-endpoint} is
\[
 \boxed{B=-\frac1\pi}.
\]

\section{Hypergeometric and elliptic-integral calculations}
\label{app:elliptic-algebra}

This appendix supplies the coefficient comparisons, integral reductions, and
algebraic simplifications used in Section~\ref{sec:independent-coeff}.  The
source formulae are cited at the points where they enter.  Recall that the
present paper uses the parameter, whereas DLMF uses the modulus:
\[
 K(q)=\mathsf K(\sqrt q),\qquad E(q)=\mathsf E(\sqrt q).
\]

\subsection{Coefficient proof of \eqref{eq:ell-H1}}

The hypergeometric representations of the complete elliptic integrals are
DLMF \cite[\S19.5, Eqs.~(19.5.1)--(19.5.2)]{DLMF}.  In the present parameter
notation, put
\[
 a_n:=\frac{(\frac12)_n^2}{(n!)^2}
     =\frac{\binom{2n}{n}^2}{16^n},
 \qquad
 \widehat K(z):=\frac2\pi K(z),\qquad
 \widehat E(z):=\frac2\pi E(z).
\]
Since
\[
 \frac{(-\frac12)_n}{(\frac12)_n}=-\frac1{2n-1}\qquad(n\ge1),
\]
the two Maclaurin series are
\[
 \widehat K(z)=\sum_{n\ge0}a_nz^n,
 \qquad
 \widehat E(z)=1-\sum_{n\ge1}\frac{a_n}{2n-1}z^n.
\]
The constant coefficient of $2\widehat E-(1-z)\widehat K$ is one.  For
$n\ge1$, its coefficient of $z^n$ is
\begin{align*}
 -\frac{2a_n}{2n-1}-a_n+a_{n-1}
 &=a_{n-1}-\frac{2n+1}{2n-1}a_n\\
 &=a_{n-1}\left(1-
   \frac{(2n+1)(2n-1)}{4n^2}\right)
 =\frac{a_{n-1}}{4n^2},
\end{align*}
where $a_n/a_{n-1}=((2n-1)/(2n))^2$ was used.  On the other hand,
\[
 \frac{(-\frac12)_n^2}{(n!)^2}
 =\frac1{4n^2}\frac{(\frac12)_{n-1}^2}{((n-1)!)^2}
 =\frac{a_{n-1}}{4n^2}.
\]
Thus the coefficients agree with those of
${}_2F_1(-1/2,-1/2;1;z)$, proving \eqref{eq:ell-H1} first for $|z|<1$.
The principal Gauss function on the left and the principal elliptic-integral
expression on the right are separately holomorphic on the connected cut plane
$\C\setminus[1,\infty)$.  Since they agree on $|z|<1$, the identity theorem
proves their equality throughout that cut plane.

\subsection{Euler-integral proof of \eqref{eq:ell-H2}}

DLMF \cite[\S15.6, Eq.~(15.6.1)]{DLMF} writes Euler's integral for Olver's
regularized hypergeometric function.  Multiplication by $\Gamma(c)$ gives the
ordinary Gauss function in the form
\[
 {}_2F_1(a,b;c;z)
 =\frac{\Gamma(c)}{\Gamma(b)\Gamma(c-b)}
  \int_0^1t^{b-1}(1-t)^{c-b-1}(1-zt)^{-a}\,\dd t.
\]
For $a=-1/2$, $b=3/2$, and $c=3$, the prefactor is
\[
 \frac{\Gamma(3)}{\Gamma(3/2)^2}=\frac8\pi.
\]
The substitution $t=\sin^2\theta$ therefore gives
\begin{equation}\label{eq:app-H2-integral}
 {}_2F_1\!\left(-\frac12,\frac32;3;z\right)
 =\frac{16}{\pi}I(z),
 \quad
 I(z):=\int_0^{\pi/2}\sin^2\theta\cos^2\theta
        \sqrt{1-z\sin^2\theta}\,\dd\theta.
\end{equation}
Initially take $|z|<1$.  Define
\[
 J_p(z):=\int_0^{\pi/2}
 \frac{\sin^{2p}\theta}{\sqrt{1-z\sin^2\theta}}\,\dd\theta,
 \qquad p\ge0.
\]
Then $J_0=K$ and
\begin{equation}\label{eq:app-J1}
 E(z)=K(z)-zJ_1(z),\qquad J_1(z)=\frac{K(z)-E(z)}z.
\end{equation}
For $p\ge1$, the function
\[
 \sin^{2p-1}\theta\cos\theta\sqrt{1-z\sin^2\theta}
\]
vanishes at both endpoints.  Integrating its derivative gives
\begin{equation}\label{eq:app-J-recurrence}
 (2p-1)J_{p-1}-2p(1+z)J_p+(2p+1)zJ_{p+1}=0.
\end{equation}
The cases $p=1,2$ yield
\[
 J_2=\frac{2(1+z)J_1-K}{3z},
 \qquad
 J_3=\frac{4(1+z)J_2-3J_1}{5z}.
\]
Expanding the numerator in the integral defining $I$ gives
\[
 I=J_1-(1+z)J_2+zJ_3
   =\frac15\bigl(2J_1-(1+z)J_2\bigr).
\]
Substitution of $J_2$ and then \eqref{eq:app-J1} gives
\begin{align*}
 I(z)
 &=\frac1{15z}\left[-2(z^2-z+1)J_1+(1+z)K\right]\\
 &=\frac1{15z^2}\left[
 2(z^2-z+1)E(z)-(z^2-3z+2)K(z)\right].
\end{align*}
Equation \eqref{eq:app-H2-integral} is now exactly \eqref{eq:ell-H2}.  In
particular, its square bracket is
\[
 \frac{15\pi}{16}z^2+O(z^3),
\]
so the displayed $z^{-2}$ in \eqref{eq:ell-H2} has a removable singularity at
zero.  The identity theorem again extends the equality from $|z|<1$ to the
common principal cut plane.

\subsection{Reduction of $F$ to $K$ and $E$}

Put
\[
 R_2(z):=z^2-z+1,\qquad S_2(z):=z^2-3z+2.
\]
Replacing $z$ by $z/16$ in \eqref{eq:F-def} and then using
\eqref{eq:ell-H1}--\eqref{eq:ell-H2} gives
\begin{align*}
 F\!\left(\frac z{16}\right)
 =\frac{64}{z^2}\biggl[&1+\frac{3z}{8}-\frac4\pi E(z)
 +\frac{2(1-z)}\pi K(z)\\
 &-\frac{2}{15\pi z}\bigl(2R_2(z)E(z)-S_2(z)K(z)\bigr)\biggr].
\end{align*}
The $E$-coefficient inside the square brackets is
\[
 -\frac4\pi-\frac{4R_2(z)}{15\pi z}
 =-\frac{4(z^2+14z+1)}{15\pi z},
\]
while the $K$-coefficient is
\[
 \frac{2(1-z)}\pi+\frac{2S_2(z)}{15\pi z}
 =-\frac{4(z-1)(7z+1)}{15\pi z}.
\]
Finally,
\[
 \frac{64}{z^2}\left(1+\frac{3z}{8}\right)
 =\frac8{15\pi z^3}\,\pi(45z^2+120z).
\]
These three identities give \eqref{eq:F-elliptic}.

\subsection{Reciprocal parameters, boundary signs, and the functions $U,V$}

The source of \eqref{eq:K-boundary}--\eqref{eq:E-boundary} is DLMF
\cite[\S19.7(i), Eq.~(19.7.3)]{DLMF}, which is written in the modulus.  In
that notation, for $\re k>0$,
\begin{align*}
 \mathsf K(1/k)
 &=k\bigl(\mathsf K(k)\mp i\mathsf K(k')\bigr),\\
 \mathsf E(1/k)
 &=k^{-1}\bigl(\mathsf E(k)\pm i\mathsf E(k')
 -(k')^2\mathsf K(k)\mp ik^2\mathsf K(k')\bigr),
\end{align*}
where the upper DLMF signs apply for $\im k^2>0$ and the lower signs for
$\im k^2<0$.

For a boundary point $x_\pm=1/q\pm i0$, set
\[
 w_\pm:=\sqrt{x_\pm},\qquad k_\pm:=w_\pm^{-1}.
\]
Then $k_\pm^2=x_\pm^{-1}=q\mp i0$: taking a reciprocal reverses the side of
the real axis.  Hence the upper boundary value $x_+$ requires the lower DLMF
signs, and the lower boundary value $x_-$ requires the upper DLMF signs.  With
$k'=\sqrt{1-q}$ and with
$K(q)=\mathsf K(\sqrt q)$, $E(q)=\mathsf E(\sqrt q)$, this gives precisely
\eqref{eq:K-boundary}--\eqref{eq:E-boundary}.

Since $16r=1/q$, substitution of $z=1/q$ into \eqref{eq:F-elliptic} gives
\begin{align*}
 F_\pm(r)=\frac{8q^3}{15\pi}\biggl[&
 \pi\left(\frac{45}{q^2}+\frac{120}{q}\right)
 -\frac{32(1+14q+q^2)}{q^2}E_\pm(1/q)\\
 &-\frac{32(1-q)(q+7)}{q^2}K_\pm(1/q)\biggr].
\end{align*}
For the real $K(q)$-coefficient one uses
\[
 (1+14q+q^2)(1-q)-q(1-q)(q+7)
 =(1-q)(1+7q),
\]
whereas the imaginary $K(1-q)$-coefficient uses
\[
 q\bigl((1+14q+q^2)+(1-q)(q+7)\bigr)=8q(1+q).
\]
After the common factor $\sqrt q$ is extracted, the real and imaginary parts
are exactly $U(q)$ and $\pm32V(q)$ in
\eqref{eq:U-def}--\eqref{eq:V-def}.  This proves
\eqref{eq:F-boundary-UV}.  Since $r=1/(16q)$,
\[
 r\left|\frac{8\sqrt q}{15\pi}\right|^2
 =\frac4{225\pi^2},
\]
which proves \eqref{eq:modulus-UV}.

Finally, as $q\uparrow1$,
\[
 E(q)\to1,\quad (1-q)K(q)\to0,\quad
 K(1-q)\to\frac\pi2,\quad E(1-q)\to\frac\pi2.
\]
The first and last two limits follow directly from the defining integrals;

for the remaining limit, the integral definition gives directly
\[
 0\le(1-q)K(q)
 \le\frac\pi2\sqrt{1-q}\longrightarrow0,
 \qquad 0<q<1,
\]
since $1-q\sin^2\theta\ge1-q$.
  Substitution gives
$U(q)\to165\pi-512$ and $V(q)\to0$, as asserted in the text.

\subsection{Differentiation of $U$ and $V$}

DLMF \cite[\S19.4(i), Eqs.~(19.4.1)--(19.4.2)]{DLMF} differentiates with
respect to the modulus $k$.  Applying the chain rule to $q=k^2$ gives
\[
 K'(q)=\frac{E(q)}{2q(1-q)}-\frac{K(q)}{2q},
 \qquad
 E'(q)=\frac{E(q)-K(q)}{2q}.
\]
For
\[
 R_U(q):=(1-q)(1+7q),\qquad R_E(q):=1+14q+q^2,
\]
we have $R_U'(q)=6-14q$ and $R_E'(q)=14+2q$.  Differentiating
\eqref{eq:U-def} and inserting the preceding derivative formulae gives
\begin{align*}
 U'(q)
 ={}&\frac{45\pi(1+8q)}{2\sqrt q}
 +\left(32R_U'-\frac{16R_U}{q}+\frac{16R_E}{q}\right)K(q)\\
 &+\left(\frac{16R_U}{q(1-q)}-32R_E'-\frac{16R_E}{q}\right)E(q)\\
 ={}&\frac{45\pi(1+8q)}{2\sqrt q}
 +320(1-q)K(q)-80(q+7)E(q)\\
 ={}&-\frac5{2\sqrt q}\left[
 32\sqrt q\bigl((q+7)E(q)-4(1-q)K(q)\bigr)
 -9\pi(8q+1)\right].
\end{align*}
This proves \eqref{eq:Uprime} and \eqref{eq:H-def}.

For the derivative of $V$, abbreviate
$K_c=K(1-q)$ and $E_c=E(1-q)$.  The chain rule gives
\[
 \frac{\dd E_c}{\dd q}=\frac{K_c-E_c}{2(1-q)},
 \qquad
 \frac{\dd K_c}{\dd q}
 =-\frac{E_c}{2q(1-q)}+\frac{K_c}{2(1-q)}.
\]
Consequently
\begin{align*}
 V'(q)={}&(14+2q)E_c
 +R_E(q)\frac{K_c-E_c}{2(1-q)}
 -8(1+2q)K_c\\
 &-8q(1+q)\left(
 -\frac{E_c}{2q(1-q)}+\frac{K_c}{2(1-q)}\right)\\
 ={}&\frac52(q+7)E_c-\frac52(5q+3)K_c,
\end{align*}
which is \eqref{eq:Vprime}.

\subsection{The identity for $H(t^2)$}

The Maclaurin series are DLMF \cite[\S19.5]{DLMF},
Eqs.~(19.5.1)--(19.5.2), written in parameter notation.  Dividing them by
$\pi$ gives
\[
 \frac{K(t^2)}\pi=\frac12\sum_{j\ge0}a_jt^{2j},
 \qquad
 \frac{E(t^2)}\pi=\frac12\left(1-\sum_{j\ge1}
 \frac{a_j}{2j-1}t^{2j}\right).
\]
Substitution into \eqref{eq:H-def} yields
\begin{align*}
 \frac{H(t^2)}\pi={}&16t(t^2+7)\left(1-\sum_{j\ge1}
 \frac{a_j}{2j-1}t^{2j}\right)
 -64t(1-t^2)\sum_{j\ge0}a_jt^{2j}\\
 &-9(8t^2+1).
\end{align*}
The coefficients of $1,t,t^2,t^3$ form
$P_H(t)=-9+48t-72t^2+36t^3$.  For $n\ge2$, the coefficient of
$t^{2n+1}$ is
\[
 -\frac{112a_n}{2n-1}-\frac{16a_{n-1}}{2n-3}
 -64a_n+64a_{n-1}.
\]
Since $a_n=a_{n-1}((2n-1)/(2n))^2$, this reduces to
\[
 -\frac{36a_{n-1}}{n^2(2n-3)},
\]
which proves \eqref{eq:H-PT}.

\subsection{Positive coefficients of $D_0$}

Substitute the same series into
$D_0(s)=(8-5s)K(s)-(8-s)E(s)$.  The constant, linear, and quadratic
coefficients vanish.  For $n\ge3$, the coefficient of $s^n$, divided by
$\pi/2$, is
\[
 8a_n-5a_{n-1}+\frac{8a_n}{2n-1}-\frac{a_{n-1}}{2n-3}.
\]
Using the ratio $a_n/a_{n-1}$ simplifies it to
\[
 \frac{6(n-1)(n-2)}{n(2n-3)}a_{n-1}>0,
\]
which is \eqref{eq:D0-positive}.

\subsection{Conversion of the complementary-modulus expansions}

DLMF \cite[\S19.12, Eqs.~(19.12.1)--(19.12.3)]{DLMF} gives the convergent
series, in modulus notation,
\begin{align*}
 \mathsf K(k)
 &=\sum_{m\ge0}\frac{(\frac12)_m^2}{(m!)^2}(k')^{2m}
   \left(\log\frac1{k'}+\delta_m\right),\\
 \mathsf E(k)
 &=1+\frac12\sum_{m\ge0}
 \frac{(\frac12)_m(\frac32)_m}{(2)_m m!}(k')^{2m+2}
 \left(\log\frac1{k'}+\delta_m
 -\frac1{(2m+1)(2m+2)}\right),
\end{align*}
where
\[
 \delta_m:=\psi(m+1)-\psi\!\left(m+\frac12\right).
\]
The digamma special values in DLMF
\cite[\S5.4(ii), Eqs.~(5.4.14)--(5.4.15)]{DLMF} give
\[
 \delta_m=2\log2-2(H_{2m}-H_m)=2\log2-d_m.
\]
To obtain $K(1-q)$ and $E(1-q)$ in the paper's parameter notation, set
$k=\sqrt{1-q}$, so $k'=\sqrt q$.  Since
\[
 \frac{(\frac12)_n^2}{(n!)^2}=a_n,
 \qquad
 \log\frac1{\sqrt q}+\delta_n
 =\log\frac4{\sqrt q}-d_n=L_q-d_n,
\]
the first DLMF series becomes \eqref{eq:K-complementary}.

For the second series put $n=m+1$.  Its coefficient becomes
\begin{align*}
 \frac12\frac{(\frac12)_{n-1}(\frac32)_{n-1}}
 {(2)_{n-1}(n-1)!}
 &=\frac{(\frac12)_{n-1}(\frac12)_n}{n!(n-1)!}\\
 &=\frac{2n}{2n-1}a_n.
\end{align*}
The expression in its parentheses is initially
\[
 L_q-d_{n-1}-\frac1{2n(2n-1)}.
\]
But
\begin{align*}
 d_n-d_{n-1}
 &=2\left(\frac1{2n-1}+\frac1{2n}-\frac1n\right)
 =\frac1{n(2n-1)}
 =\frac2{2n(2n-1)}.
\end{align*}
Therefore
\[
 L_q-d_{n-1}-\frac1{2n(2n-1)}
 =L_q-d_n+\frac1{2n(2n-1)},
\]
which proves \eqref{eq:E-complementary}, including the apparently changed
sign of the last term.

\subsection{The explicit bounds at $q=1/8$}

For $n\ge1$,
\[
 \frac{d_n}{2}=\sum_{j=n+1}^{2n}\frac1j
 <\int_n^{2n}\frac{\dd x}{x}=\log2.
\]
Thus $d_n<2\log2<L=(7/2)\log2$, and every summand in
\eqref{eq:E-complementary} at $q=1/8$ is positive.  The $n=1$ summand is
\[
 \frac{2a_1}{8}\left(L-d_1+\frac12\right)
 =\frac1{16}\left(L-\frac12\right),
\]
because $a_1=1/4$ and $d_1=1$.  Keeping this term and using
$L>L_-$ proves \eqref{eq:E78-lower}.

For the upper bound, \eqref{eq:K-complementary} gives
\[
 K(7/8)=L+\frac{L-1}{32}
 +\sum_{n\ge2}a_n8^{-n}(L-d_n).
\]
Here $0<a_n\le1$ and $0<L-d_n<L$, hence
\[
 0<\sum_{n\ge2}a_n8^{-n}(L-d_n)
 <L\sum_{n\ge2}8^{-n}=\frac{L}{56}.
\]
Replacing $L$ by its upper rational bound $L_+$ in the three increasing
expressions gives \eqref{eq:K78-upper}.

\section{Convergent hypergeometric expansions at infinity and the leading
constant}
\label{app:infinity}

This appendix derives the expansion used in
\eqref{eq:F-infinity}--\eqref{eq:Phi-infinity} directly from a convergent
series in $1/w$.  Throughout, $|w|>1$ and
$|\Arg(-w)|<\pi$, and
\[
 (-w)^\alpha:=\exp\{\alpha\Log(-w)\}.
\]
The same convention is used after $w=16z$.  Thus the branches below are the
ones induced by the principal continuation on the slit plane.

\subsection{The integer-difference connection formula}

DLMF \cite[\S15.8(ii), Eq.~(15.8.8)]{DLMF} states the formula for Olver's
regularized function
$\mathbf F(a,b;c;w)={}_2F_1(a,b;c;w)/\Gamma(c)$.  After multiplication by
$\Gamma(c)$, its ordinary Gauss-function form is, for
$\ell\in\N_0$,
\begin{align}
 {}_2F_1(a,a+\ell;c;w)
 ={}&\frac{\Gamma(c)(-w)^{-a}}{\Gamma(a+\ell)}
 \sum_{k=0}^{\ell-1}
 \frac{(a)_k(\ell-k-1)!}{k!\,\Gamma(c-a-k)}w^{-k}
 \notag\\
 &+\frac{\Gamma(c)(-w)^{-a}}{\Gamma(a)}
 \sum_{k=0}^{\infty}
 \frac{(a+\ell)_k(-1)^k w^{-k-\ell}}
 {k!(k+\ell)!\,\Gamma(c-a-k-\ell)}
 \Xi_{\ell,k}(a,c;w),
 \label{eq:infinity-connection}
\end{align}
where the first sum is empty when $\ell=0$, and
\begin{align}
 \Xi_{\ell,k}(a,c;w):={}&\Log(-w)+\psi(k+1)+\psi(k+\ell+1)
 \notag\\
 &-\psi(a+k+\ell)-\psi(c-a-k-\ell).
 \label{eq:Xi-infinity}
\end{align}
Here $\psi=\Gamma'/\Gamma$.  The displayed gamma and digamma factors are assumed finite, as they are in the two applications below.
Equations
\eqref{eq:infinity-connection}--\eqref{eq:Xi-infinity} are exact convergent
series, not merely formal asymptotic expansions.  Their convergence is locally uniform on $|w|>1$, $|\Arg(-w)|<\pi$.

Put
\[
 H_1(w):={}_2F_1\!\left(-\frac12,-\frac12;1;w\right),
 \qquad
 H_2(w):={}_2F_1\!\left(-\frac12,\frac32;3;w\right).
\]
For $H_1$, take $a=-1/2$, $c=1$, and $\ell=0$ in
\eqref{eq:infinity-connection}.  This gives the full convergent expansion
\begin{align}
 H_1(w)={}&\frac{(-w)^{1/2}}{\Gamma(-1/2)}
 \sum_{k=0}^{\infty}
 \frac{(-1/2)_k(-1)^k w^{-k}}
 {(k!)^2\Gamma(3/2-k)}
 \notag\\[-2mm]
 &\qquad\times\left[
 \Log(-w)+2\psi(k+1)-\psi\!\left(k-\frac12\right)
 -\psi\!\left(\frac32-k\right)
 \right].
 \label{eq:H1-infinity-series}
\end{align}
For $H_2$, take $a=-1/2$, $c=3$, and $\ell=2$.  The two terms in the
finite sum and the infinite logarithmic series give
\begin{align}
 H_2(w)={}&(-w)^{1/2}\left(\frac{32}{15\pi}
 -\frac{8}{3\pi}w^{-1}\right)
 \notag\\
 &-\frac{(-w)^{1/2}}{\sqrt\pi}
 \sum_{k=0}^{\infty}
 \frac{(3/2)_k(-1)^k w^{-k-2}}
 {k!(k+2)!\Gamma(3/2-k)}
 \notag\\[-2mm]
 &\qquad\times\left[
 \Log(-w)+\psi(k+1)+\psi(k+3)
 -\psi\!\left(k+\frac32\right)
 -\psi\!\left(\frac32-k\right)
 \right].
 \label{eq:H2-infinity-series}
\end{align}
For clarity, the coefficients of the finite sum in
\eqref{eq:H2-infinity-series} are
\begin{align*}
 \frac{\Gamma(3)}{\Gamma(3/2)\Gamma(7/2)}
 &=\frac{32}{15\pi},\\
 \frac{\Gamma(3)}{\Gamma(3/2)}
 \frac{-1/2}{\Gamma(5/2)}
 &=-\frac{8}{3\pi}.
\end{align*}
This displays explicitly why the dominant coefficient is
$32/(15\pi)$ and also records the next algebraic term.

\subsection{Direct convergence verification}

For completeness, we now verify the convergence of the two infinite
series in \eqref{eq:H1-infinity-series} and
\eqref{eq:H2-infinity-series} directly from their coefficients.  This
also supplies the uniform tail estimates used below.  Put
\begin{align}
 \mathfrak a_{1,k}
 &:=\frac{(-\tfrac12)_k(-1)^k}
          {(k!)^2\Gamma(\tfrac32-k)},
 &
 \delta_{1,k}
 &:=2\psi(k+1)-\psi(k-\tfrac12)-\psi(\tfrac32-k),
 \label{eq:D-a1-delta1-def}\\
 \mathfrak a_{2,k}
 &:=\frac{(\tfrac32)_k(-1)^k}
          {k!(k+2)!\Gamma(\tfrac32-k)},
 &
 \delta_{2,k}
 &:=\psi(k+1)+\psi(k+3)-\psi(k+\tfrac32)
       -\psi(\tfrac32-k).
 \label{eq:D-a2-delta2-def}
\end{align}
Thus the infinite sums occurring in $H_1$ and $H_2$ are, respectively,
\begin{equation}
 \sum_{k\ge0}\mathfrak a_{1,k}w^{-k}
       \bigl(\Log(-w)+\delta_{1,k}\bigr),
 \qquad
 \sum_{k\ge0}\mathfrak a_{2,k}w^{-k-2}
       \bigl(\Log(-w)+\delta_{2,k}\bigr).
 \label{eq:D-two-series-compact}
\end{equation}

The reflection identities for the gamma and digamma functions give,
for every integer $k\ge0$,
\begin{equation}
 \frac{1}{\Gamma(\tfrac32-k)}
 = \frac{(-1)^{k+1}}{\pi}\Gamma(k-\tfrac12),
 \qquad
 \psi(\tfrac32-k)=\psi(k-\tfrac12).
 \label{eq:D-reflection-identities}
\end{equation}
Indeed, the first equality follows from
$\Gamma(z)\Gamma(1-z)=\pi/\sin(\pi z)$ with
$z=\tfrac32-k$, while the second follows from
$\psi(1-z)-\psi(z)=\pi\cot(\pi z)$ with
$z=k-\tfrac12$, for which the cotangent term vanishes.

Using $(-\tfrac12)_k=\Gamma(k-\tfrac12)/\Gamma(-\tfrac12)$,
$\Gamma(-\tfrac12)=-2\sqrt\pi$, and
\eqref{eq:D-reflection-identities}, we obtain the exact identity
\begin{equation}
 \mathfrak a_{1,k}
 =-\frac{\Gamma(k-\tfrac12)^2}
          {\pi\Gamma(-\tfrac12)\Gamma(k+1)^2}
 =\frac{1}{2\pi^{3/2}}
   \left(\frac{\Gamma(k-\tfrac12)}{\Gamma(k+1)}\right)^2.
 \label{eq:D-a1-exact}
\end{equation}
The fixed-shift gamma-ratio estimate \cite[\S5.11(iii)]{DLMF}
\begin{equation}
 \frac{\Gamma(k+\alpha)}{\Gamma(k+\beta)}
 =k^{\alpha-\beta}\bigl(1+O(k^{-1})\bigr)
 \qquad (k\to\infty)
 \label{eq:D-gamma-ratio}
\end{equation}
therefore yields
\begin{equation}
 \mathfrak a_{1,k}
 =\frac{1}{2\pi^{3/2}}k^{-3}
   \bigl(1+O(k^{-1})\bigr).
 \label{eq:D-a1-asymptotic}
\end{equation}
Moreover, the second identity in
\eqref{eq:D-reflection-identities} reduces the digamma correction to
\begin{equation}
 \delta_{1,k}=2\psi(k+1)-2\psi(k-\tfrac12).
 \label{eq:D-delta1-reduced}
\end{equation}
For each fixed $s$, DLMF \cite[Eq.~(5.11.2)]{DLMF}, expanded at $k+s$, gives
\begin{equation}
 \psi(k+s)=\log k+\frac{s-\tfrac12}{k}+O(k^{-2}),
 \qquad k\to\infty.
 \label{eq:D-digamma-fixed-shift}
\end{equation}
Substitution of $s=1$ and $s=-\tfrac12$ in
\eqref{eq:D-digamma-fixed-shift} gives
\begin{equation}
 \delta_{1,k}=\frac{3}{k}+O(k^{-2}).
 \label{eq:D-delta1-asymptotic}
\end{equation}

The second series is handled in the same direct manner, but we record
all exponents because they determine the strength of the remainder.
From $(\tfrac32)_k=\Gamma(k+\tfrac32)/\Gamma(\tfrac32)$ and
\eqref{eq:D-reflection-identities},
\begin{equation}
 \mathfrak a_{2,k}
 =-\frac{\Gamma(k+\tfrac32)\Gamma(k-\tfrac12)}
          {\pi\Gamma(\tfrac32)\Gamma(k+1)\Gamma(k+3)}.
 \label{eq:D-a2-exact}
\end{equation}
Since
\begin{equation}
 \frac{\Gamma(k+\tfrac32)}{\Gamma(k+1)}
 =k^{1/2}\bigl(1+O(k^{-1})\bigr),
 \qquad
 \frac{\Gamma(k-\tfrac12)}{\Gamma(k+3)}
 =k^{-7/2}\bigl(1+O(k^{-1})\bigr),
 \label{eq:D-a2-gamma-ratios}
\end{equation}
we have
\begin{equation}
 \mathfrak a_{2,k}
 =-\frac{1}{\pi\Gamma(\tfrac32)}k^{-3}
   \bigl(1+O(k^{-1})\bigr)
 =-\frac{2}{\pi^{3/2}}k^{-3}
   \bigl(1+O(k^{-1})\bigr).
 \label{eq:D-a2-asymptotic}
\end{equation}
Likewise,
\begin{equation}
 \delta_{2,k}
 =\psi(k+1)+\psi(k+3)-\psi(k+\tfrac32)-\psi(k-\tfrac12),
 \label{eq:D-delta2-reduced}
\end{equation}
and \eqref{eq:D-digamma-fixed-shift}, applied with
$s=1,3,\tfrac32,-\tfrac12$, gives
\begin{equation}
 \delta_{2,k}=\frac{3}{k}+O(k^{-2}).
 \label{eq:D-delta2-asymptotic}
\end{equation}

We can now read off convergence without appealing to a differential
equation.  For each fixed $w$ with $|w|>1$, the terms in the two sums
in \eqref{eq:D-two-series-compact} are bounded, for all sufficiently
large $k$, by
\begin{equation}
 C_w k^{-3}|w|^{-k},
 \qquad
 C_w k^{-3}|w|^{-k-2},
 \label{eq:D-pointwise-majorants}
\end{equation}
respectively.  Hence both series converge absolutely.  More precisely,

after putting $q=1/w$, the $j$th expression is
$q^{\ell_j}(\Log(-w)A_j(q)+B_j(q))$, where $\ell_1=0$, $\ell_2=2$,
$A_j(q)=\sum_{k\ge0}\mathfrak a_{j,k}q^k$, and
$B_j(q)=\sum_{k\ge0}\mathfrak a_{j,k}\delta_{j,k}q^k$.
  By
\eqref{eq:D-a1-asymptotic}--\eqref{eq:D-delta2-asymptotic}, the
respective coefficients have orders $k^{-3}$ and $k^{-4}$;
Cauchy--Hadamard therefore gives radius of convergence exactly $1$
in the variable $q$.

The convergence is also locally uniform on
\begin{equation}
 \{w:|w|>1,\ |\Arg(-w)|<\pi\}.
 \label{eq:D-infinity-domain}
\end{equation}
Indeed, if $K$ is a compact subset of this domain, then there are
constants $R_K>1$ and $M_K<\infty$ such that
\begin{equation}
 |w|\ge R_K,
 \qquad
 |\Log(-w)|\le M_K
 \qquad (w\in K).
 \label{eq:D-compact-log-bounds}
\end{equation}
The two summands are consequently dominated on $K$ by the summable
majorants
\begin{equation}
 C_K k^{-3}R_K^{-k},
 \qquad
 C_K k^{-3}R_K^{-k-2}.
 \label{eq:D-compact-majorants}
\end{equation}
The Weierstrass $M$-test proves local uniform convergence; since every
summand is holomorphic there, both sums are holomorphic on
\eqref{eq:D-infinity-domain}.

Finally, the same coefficient estimates give the uniform tail bounds
needed for the expansion at infinity.  By increasing the constants to cover the finitely many small indices, one has
$|\mathfrak a_{j,k}|\le C_jk^{-3}$ and $|\delta_{j,k}|\le D_jk^{-1}$ for every $k\ge1$; the $k=0$ terms are retained separately.  Fix $R>1$.  On the slit plane
with $|w|\ge R$,
\begin{equation}
 |\Log(-w)|\le \log|w|+\pi.
 \label{eq:D-global-log-bound}
\end{equation}
After the $k=0$ term has been removed from each infinite sum, the
remaining parts of \eqref{eq:D-two-series-compact} satisfy
\begin{align}
 \left|
   \sum_{k\ge1}\mathfrak a_{1,k}w^{-k}
       \bigl(\Log(-w)+\delta_{1,k}\bigr)
 \right|
 &=O_R\bigl(|w|^{-1}(1+\log|w|)\bigr),
 \label{eq:D-H1-inner-tail}\\
 \left|
   \sum_{k\ge1}\mathfrak a_{2,k}w^{-k-2}
       \bigl(\Log(-w)+\delta_{2,k}\bigr)
 \right|
 &=O_R\bigl(|w|^{-3}(1+\log|w|)\bigr).
 \label{eq:D-H2-inner-tail}
\end{align}
For example,
\[
 \sum_{k\ge1} k^{-3}|w|^{-k}
 \le |w|^{-1}\sum_{k\ge1}k^{-3}R^{-(k-1)},
\]
and the series on the right is a constant depending only on $R$;
the second estimate is identical after factoring out $|w|^{-3}$.
Multiplication by the common outer factor $(-w)^{1/2}$ in
\eqref{eq:H1-infinity-series} and \eqref{eq:H2-infinity-series}
therefore gives

\begin{align}
 &H_1(w)+\frac{(-w)^{1/2}}\pi[\Log(-w)+4\log2-4]\notag\\
 &\qquad=O_R\bigl(|w|^{-1/2}(1+\log|w|)\bigr),
 \label{eq:D-H1-tail-bound}\\
 &H_2(w)-(-w)^{1/2}\left[\frac{32}{15\pi}-\frac8{3\pi}w^{-1}\right]\notag\\
 &\qquad+\frac{(-w)^{-3/2}}\pi\left[\Log(-w)+4\log2-\frac52\right]\notag\\
 &\qquad=O_R\bigl(|w|^{-5/2}(1+\log|w|)\bigr).
 \label{eq:D-H2-tail-bound}
\end{align}
The explicit retained terms are evaluated in the next subsection.

These are direct quantitative convergence estimates for the two
connection series.

\subsection{Extraction of the first terms}

The gamma and digamma values needed for the $k=0$ terms are
\begin{gather*}
 \Gamma\!\left(-\frac12\right)=-2\sqrt\pi,
 \quad \Gamma\!\left(\frac32\right)=\frac{\sqrt\pi}{2},
 \quad \Gamma\!\left(\frac52\right)=\frac{3\sqrt\pi}{4},
 \quad \Gamma\!\left(\frac72\right)=\frac{15\sqrt\pi}{8},\\
 \psi(1)=-\gamma,
 \quad \psi(3)=\frac32-\gamma,
 \quad
 \psi\!\left(-\frac12\right)
 =\psi\!\left(\frac32\right)
 =2-\gamma-2\log2.
\end{gather*}
Consequently, the $k=0$ bracket in
\eqref{eq:H1-infinity-series} is
\[
 \Log(-w)+2\psi(1)-\psi(-1/2)-\psi(3/2)
 =\Log(-w)+4\log2-4,
\]
whereas the $k=0$ bracket in the infinite sum of
\eqref{eq:H2-infinity-series} is
\[
 \Log(-w)+\psi(1)+\psi(3)-2\psi(3/2)
 =\Log(-w)+4\log2-\frac52.
\]
The uniform tail estimates
\eqref{eq:D-H1-tail-bound}--\eqref{eq:D-H2-tail-bound} imply, uniformly
in the slit plane as $|w|\to\infty$, that
\begin{align}
 H_1(w)={}&-\frac{(-w)^{1/2}}{\pi}
 \left(\Log(-w)+4\log2-4\right)
 +O\bigl(|w|^{-1/2}(1+\log|w|)\bigr),
 \label{eq:H1-infinity-first}\\
 H_2(w)={}&\frac{32}{15\pi}(-w)^{1/2}
 -\frac{8}{3\pi}(-w)^{1/2}w^{-1}
 \notag\\
 &-\frac1\pi(-w)^{-3/2}
 \left(\Log(-w)+4\log2-\frac52\right)
 +O\bigl(|w|^{-5/2}(1+\log|w|)\bigr).
 \label{eq:H2-infinity-first}
\end{align}
The stated uniformity is precisely the uniformity established in the
preceding direct convergence verification.

Now set $w=16z$.  Since
$(-16z)^{1/2}=4(-z)^{1/2}$ and
$\Log(-16z)=\Log(-z)+4\log2$, equations
\eqref{eq:H1-infinity-first}--\eqref{eq:H2-infinity-first} become
\begin{align}
 H_1(16z)={}&-\frac4\pi(-z)^{1/2}
 \left(\Log(-z)+8\log2-4\right)
 +O\bigl(|z|^{-1/2}(1+\log|z|)\bigr),
 \label{eq:H1-16z-infinity}\\
 H_2(16z)={}&\frac{128}{15\pi}(-z)^{1/2}
 +\frac{2}{3\pi}(-z)^{-1/2}
 \notag\\
 &-\frac1{64\pi}(-z)^{-3/2}
 \left(\Log(-z)+8\log2-\frac52\right)
 +O\bigl(|z|^{-5/2}(1+\log|z|)\bigr).
 \label{eq:H2-16z-infinity}
\end{align}
The sign of the second term in \eqref{eq:H2-16z-infinity} follows from
$(-z)^{1/2}z^{-1}=-(-z)^{-1/2}$.

\subsection{Substitution into $F$}

Insert \eqref{eq:H1-16z-infinity}--\eqref{eq:H2-16z-infinity} into
\eqref{eq:F-def}.  The three contributions are
\begin{align*}
 \frac{1+6z}{4z^2}
 &=\frac3{2z}+\frac1{4z^2},\\
 -\frac{H_1(16z)}{4z^2}
 &=\frac1\pi(-z)^{-3/2}
 \left(\Log(-z)+8\log2-4\right)
 +O\bigl(|z|^{-5/2}(1+\log|z|)\bigr),\\
 -\frac{H_2(16z)}{2z}
 &=\frac{64}{15\pi}(-z)^{-1/2}
 +\frac1{3\pi}(-z)^{-3/2}
 +O\bigl(|z|^{-5/2}(1+\log|z|)\bigr).
\end{align*}
Thus the more detailed form of the infinity expansion is
\begin{align}
 F(z)={}&\frac{64}{15\pi}(-z)^{-1/2}+\frac3{2z}
 +\frac1\pi(-z)^{-3/2}
 \left(\Log(-z)+8\log2-\frac{11}{3}\right)
 \notag\\
 &+\frac1{4z^2}
 +O\bigl(|z|^{-5/2}(1+\log|z|)\bigr).
 \label{eq:F-infinity-refined}
\end{align}
Discarding the displayed terms of order
$O(|z|^{-3/2}(1+\log|z|))$ proves \eqref{eq:F-infinity}.

\subsection{Squaring the modulus and identifying the constant}

Set
\[
 \lambda_\infty:=\frac{64}{15\pi},
 \qquad u(z):=(-z)^{-1/2},
 \qquad v(z):=\frac3{2z}.
\]
The coarser form \eqref{eq:F-infinity} says
\[
 F(z)=\lambda_\infty u(z)+v(z)+\mathcal R_\infty(z),
 \qquad
 \mathcal R_\infty(z)
 =O\bigl(|z|^{-3/2}(1+\log|z|)\bigr).
\]
Since $|u(z)|=|z|^{-1/2}$, direct multiplication gives
\begin{align}
 |z|\,|F(z)|^2
 ={}&|z|\bigl(
 \lambda_\infty^2|u|^2
 +2\lambda_\infty\Re(u\overline v)
 +2\lambda_\infty\Re(u\overline{\mathcal R_\infty})
 \notag\\
 &\hspace{18mm}+|v|^2
 +2\Re(v\overline{\mathcal R_\infty})
 +|\mathcal R_\infty|^2\bigr).
 \label{eq:modulus-infinity-expanded}
\end{align}
The six terms on the right have respective orders
\[
 \begin{aligned}
 &\lambda_\infty^2,\qquad
 O(|z|^{-1/2}),\\
 &O\bigl(|z|^{-1}(1+\log|z|)\bigr),\qquad
 O(|z|^{-1}),\\
 &O\bigl(|z|^{-3/2}(1+\log|z|)\bigr),\qquad
 O\bigl(|z|^{-2}(1+\log|z|)^2\bigr).
 \end{aligned}
\]
All bounds are uniform in $z\in\D$.  In particular, every term except the
first is $O(|z|^{-1/2})$, and therefore
\[
 |z|\,|F(z)|^2
 =\lambda_\infty^2+O(|z|^{-1/2})
 =\frac{4096}{225\pi^2}+O(|z|^{-1/2}).
\]
This proves \eqref{eq:Phi-infinity} and identifies the leading constant as
\[
 \boxed{L_\infty=\frac{4096}{225\pi^2}}.
\]
Finally, the rational bounds used after \eqref{eq:constants-exact} give
$0<165\pi-512<46/7<32$.  Hence
\[
 \sigma=\frac{4(165\pi-512)^2}{225\pi^2}
 <\frac{4\cdot32^2}{225\pi^2}
 =L_\infty,
\]
which also supplies the strict comparison required in the far-tail estimates.

\section{Disk-chain proof of the monodromy gluing}
\label{app:monodromy}

This appendix supplies the two facts used in Theorem~\ref{thm:S-delta}: the
standard dented domain is simply connected, and the analytic function element
determined by the power series of $S$ at zero can be continued along every path
by a finite chain of open disks.

Retain the notation
\begin{align}
 \B&:=\{y:|y|<\sigma\},\label{eq:app-B}\\
 A_\phi&:=\{y:y\ne\sigma,\ |\Arg(y-\sigma)|>\phi\},\label{eq:app-Aphi}\\
 \Delta&:=\{y:|y|<\sigma+\epsilon\}\cap A_\phi,\label{eq:app-Delta}\\
 L&:=D(\sigma,\epsilon_y)\cap A_\phi,
 \qquad D(c,r):=\{y:|y-c|<r\}.
 \label{eq:app-L}
\end{align}
We may assume $0<\epsilon_y<\sigma$.  The original power-series function on
$\B$ is denoted by $S$, and the local sectorial function on $L$ by $S_L$.

\subsection{The dented domain is star-shaped}

\begin{proposition}\label{prop:Delta-star}
For every $\sigma>0$, $\epsilon>0$, and $0<\phi<\pi/2$, the set
$\Delta(\sigma,\epsilon,\phi)$ is an open star-shaped set containing zero.
Hence it is contractible and simply connected.
\end{proposition}

\begin{proof}
Let
\[
 C_\phi=\{z:\re z\ge0,\ |\im z|\le(\tan\phi)\re z\}.
\]
This is the closed cone with vertex zero and half-angle $\phi$ about the
positive real axis.  Then $A_\phi=\sigma+(\C\setminus C_\phi)$ is open, and
$0-\sigma=-\sigma\notin C_\phi$, so $0\in\Delta$.

Take $y\in\Delta$ and set $y_t=ty$ for $0\le t\le1$,
$z=y-\sigma$, and $z_t=y_t-\sigma=tz-(1-t)\sigma$.  Clearly
$|y_t|\le|y|<\sigma+\epsilon$.  Also $y_t\ne\sigma$: if $ty=\sigma$, then for
$t<1$ the number $y-\sigma=\sigma(t^{-1}-1)$ is positive real, contrary to
$y\in A_\phi$, while $t=1$ contradicts $y\ne\sigma$.

It remains to show $z_t\notin C_\phi$.  If $\re z\le0$, then for $t<1$,
$\re z_t=t\re z-(1-t)\sigma<0$; the case $t=1$ is already known.  If
$\re z>0$, then $z\notin C_\phi$ means
\begin{equation}\label{eq:outside-cone}
 |\im z|>(\tan\phi)\re z.
\end{equation}
For a fixed $t$, if $\re z_t\le0$ there is nothing to prove.  If
$\re z_t>0$, then
\[
 |\im z_t|=t|\im z|>t(\tan\phi)\re z
 \ge(\tan\phi)(t\re z-(1-t)\sigma)
 =(\tan\phi)\re z_t.
\]
Thus $z_t\notin C_\phi$ in all cases, so the whole segment $[0,y]$ lies in
$\Delta$.  The homotopy $H(y,s)=(1-s)y$ contracts the domain to zero.
\end{proof}

\subsection{A compatible finite background cover}

Put
\begin{equation}\label{eq:app-K}
 K=\{a:|a|=\sigma,\ |a-\sigma|\ge\epsilon_y/2\}.
\end{equation}
For every $a\in K$, Theorem~\ref{thm:nonprincipal-continuation} gives an open
neighborhood $W_a$ and a function $S_a\in\mathcal O(W_a)$ such that
\begin{equation}\label{eq:Sa-agree}
 S_a=S\quad\text{on }W_a\cap\B.
\end{equation}
The compact set $K$ lies in $A_\phi$: for $|a|=\sigma$ and $a\ne\sigma$,
\[
 \re(a-\sigma)=-\frac{|a-\sigma|^2}{2\sigma}<0.
\]
For $q_a=|a-\sigma|$, choose $\rho_a>0$ so that
\begin{equation}\label{eq:rho-a-conditions}
 D(a,\rho_a)\subset W_a\cap A_\phi,\qquad \rho_a<q_a/2,
\end{equation}
and, when $q_a\ne\epsilon_y$,
\begin{equation}\label{eq:rho-a-radial}
 \rho_a<|q_a-\epsilon_y|.
\end{equation}
Choose distinct $a_1,\ldots,a_m\in K$ such that the disks
$D_j=D(a_j,\rho_j)$ cover $K$, put $S_j=S_{a_j}|_{D_j}$, and let
$N=\bigcup_{j=1}^mD_j$.

\begin{lemma}[Compatibility of circle disks]\label{lem:disk-disk}
If $D_i\cap D_j\ne\varnothing$, then $S_i=S_j$ on $D_i\cap D_j$.
\end{lemma}

\begin{proof}
Assume $i\ne j$.  Intersecting open disks imply
$|a_i-a_j|<\rho_i+\rho_j$, so for some $0<t<1$ the point
$z=(1-t)a_i+ta_j$ lies in their intersection.  Since the open disk $\B$ is
strictly convex and $a_i,a_j$ are distinct points of its boundary,
$|z|<(1-t)|a_i|+t|a_j|=\sigma$.  Thus
$D_i\cap D_j\cap\B$ is a nonempty open set.  There
$S_i=S=S_j$ by \eqref{eq:Sa-agree}.  The intersection of two disks is convex
and hence connected, so the identity theorem proves equality everywhere on
the intersection.
\end{proof}

Before comparing $D_j$ with $L$, note that if $y\in\B$ and $z=y-\sigma$, then
\[
 |\sigma+z|^2<\sigma^2\Longrightarrow
 2\sigma\re z+|z|^2<0\Longrightarrow\re z<0.
\]
Since $\phi<\pi/2$, this shows
\begin{equation}\label{eq:L-B-convex}
 L\cap\B=D(\sigma,\epsilon_y)\cap\B,
\end{equation}
a nonempty convex intersection of two disks.  The functions $S_L$ and $S$
agree first on a short real interval to the left of $\sigma$, and therefore by
the identity theorem on all of $L\cap\B$.

\begin{lemma}[Compatibility with the local sector]\label{lem:disk-local}
For every $j$, if $D_j\cap L\ne\varnothing$, then $D_j\cap L$ is connected,
$D_j\cap L\cap\B\ne\varnothing$, and $S_j=S_L$ on $D_j\cap L$.
\end{lemma}

\begin{proof}
Let $q_j=|a_j-\sigma|$.  Because $D_j\subset A_\phi$, it remains only to
compare it with $D(\sigma,\epsilon_y)$.  If $q_j>\epsilon_y$, conditions
\eqref{eq:rho-a-radial} and the reverse triangle inequality make the two disks
disjoint.  If $q_j<\epsilon_y$, the triangle inequality gives
$D_j\subset D(\sigma,\epsilon_y)$, so $D_j\subset L$.  The points
$z_t=(1-t)a_j$ lie in $D_j\cap L\cap\B$ for sufficiently small $t>0$.

If $q_j=\epsilon_y$, then
$D_j\cap L=D_j\cap D(\sigma,\epsilon_y)$ is convex.  Since
$|a_j|=\sigma$ and $|a_j-\sigma|=\epsilon_y$,
\[
 \re((a_j-\sigma)\overline{a_j})=\epsilon_y^2/2.
\]
Hence
\[
 |(1-t)a_j-\sigma|^2
 =|(a_j-\sigma)-ta_j|^2
 =\epsilon_y^2-t\epsilon_y^2+t^2\sigma^2<\epsilon_y^2
\]
for small $t>0$; choosing also $t\sigma<\rho_j$ puts $z_t$ in the triple
intersection.  In each nonempty case, $S_j=S=S_L$ on a nonempty open subset
of the connected intersection, and the identity theorem completes the proof.
\end{proof}

Thus the finite background cover
\begin{equation}\label{eq:background-cover}
 \mathcal U=\{\B,D_1,\ldots,D_m,L\}
\end{equation}
carries functions $F_{\B}=S$, $F_{D_j}=S_j$, $F_L=S_L$ that agree on every
nonempty pairwise intersection.

Since $K\subset N$, $K$ is compact, and $N$ is open,
\[
 d:=\dist(K,\C\setminus N)>0.
\]
Choose $0<\epsilon<\min\{d/2,\epsilon_y/4\}$.

\begin{lemma}[Cover of the full $\Delta$-domain]\label{lem:Delta-cover-app}
For this choice of $\epsilon$,
\[
 \Delta(\sigma,\epsilon,\phi)\subset\B\cup N\cup L.
\]
\end{lemma}

\begin{proof}
The proof is the radial projection argument in
\eqref{eq:Delta-cover}.  If $|y|<\sigma$, then $y\in\B$.  Otherwise, if
$|y-\sigma|<\epsilon_y$, then $y\in L$.  In the remaining case let
$y^*=\sigma y/|y|$.  Then $|y-y^*|<\epsilon$ and
\[
 |y^*-\sigma|>|y-\sigma|-\epsilon
 \ge\epsilon_y-\epsilon>\epsilon_y/2,
\]
so $y^*\in K$.  If $y\notin N$, then
$d\le|y-y^*|<\epsilon<d$, a contradiction.
\end{proof}

\subsection{Continuation along an arbitrary path}

\begin{proposition}\label{prop:path-continuation}
Let $\gamma:[0,1]\to\Delta$ be continuous with $\gamma(0)=0$.  The analytic
function element determined by the power series of $S$ at zero can be
continued along $\gamma$ by a finite sequence of function elements on open
disks.
\end{proposition}

\begin{proof}
For each $s\in[0,1]$, choose $U_s\in\mathcal U$ containing $\gamma(s)$.
Since $U_s\cap\Delta$ is a neighborhood of $\gamma(s)$, choose $r_s>0$ such
that
\[
 \Omega_s=D(\gamma(s),r_s)\subset U_s\cap\Delta,
 \qquad f_s=F_{U_s}|_{\Omega_s}.
\]
The sets $I_s=\gamma^{-1}(\Omega_s)$ form an open cover of $[0,1]$.  By the
Lebesgue number lemma there is $\lambda>0$ such that every subset of diameter
less than $\lambda$ lies in some $I_s$.  Choose
\[
 0=t_0<t_1<\cdots<t_n=1,
 \qquad t_k-t_{k-1}<\lambda.
\]
For every $k$, choose $s_k$ with
$[t_{k-1},t_k]\subset I_{s_k}$ and abbreviate
$\Omega_k=\Omega_{s_k}$, $f_k=f_{s_k}$, $U_k=U_{s_k}$.
Then, for $1\le k\le n-1$,
\[
 \gamma(t_k)\in\Omega_k\cap\Omega_{k+1}\subset U_k\cap U_{k+1}.
\]
Compatibility of the background functions gives
$f_k=f_{k+1}$ on this disk intersection, so the next function element is a
direct analytic continuation of the preceding one at the common point.
Finally, $0\in\Omega_1$, and the same compatibility shows that
$(\Omega_1,f_1)$ agrees near zero with the original power-series function,
even if $U_1\ne\B$.  Thus
$(\Omega_1,f_1),\ldots,(\Omega_n,f_n)$ is the required finite disk chain.
\end{proof}

\begin{theorem}[Monodromy conclusion]\label{thm:monodromy-app}
The analytic function element determined by the power series of $S$ at zero
has a unique single-valued holomorphic continuation to
$\Delta(\sigma,\epsilon,\phi)$.  It equals the original power series on
$\B$ and the local sectorial function on $L\cap\Delta$.
\end{theorem}

\begin{proof}
By Proposition~\ref{prop:Delta-star}, the domain is simply connected; by
Proposition~\ref{prop:path-continuation}, this function element can be continued
along every path from the base point.  The monodromy theorem \cite{Forster}
says that the terminal function element depends only on the endpoint and
therefore determines a unique holomorphic function
on the whole domain.  The compatible background functions already glue on
$\B\cup N\cup L$.  Their restriction to $\Delta$ agrees with the original
power series near zero, so uniqueness identifies it with the monodromy
continuation.  This proves both
stated identifications without requiring a path from $0$ inside $L$.
\end{proof}

\begin{remark}
Once all pairwise compatibilities in \eqref{eq:background-cover} have been
proved, the holomorphic gluing lemma already defines a function on
$\B\cup N\cup L$; restricting it to $\Delta$ is an alternative to invoking
monodromy.  The explicit path proof is retained to match the usual textbook
definition of analytic continuation.
\end{remark}

\section*{Acknowledgements}
The author gratefully acknowledges the use of ChatGPT 5.6 Sol (OpenAI) for
language editing and assistance in checking intermediate calculations and the
exposition.  The author independently verified all mathematical arguments and
assumes full responsibility for the content and conclusions.

\end{document}